\documentclass[11pt]{amsart}
\pdfoutput=1
\usepackage[latin1]{inputenc}
\usepackage{amsfonts,caption,comment,lscape}
\usepackage{amsmath, threeparttable,supertabular,longtable}
\usepackage{amssymb}
\usepackage{amsthm}
\usepackage[T1]{fontenc}
\usepackage[dvips]{graphicx}
\usepackage{pdfsync}
\usepackage{pifont}

\newcommand{\GL}{\mathrm{GL}}
\newcommand{\SL}{\mathrm{SL}}
\newcommand{\SU}{\mathrm{SU}}
\newcommand{\Aut}{\mathrm{Aut}}
\newcommand{\GU}{\mathrm{GU}}
\newcommand{\Sp}{\mathrm{Sp}}

\renewcommand{\O}{\mathrm{O}}

\newcommand{\PGU}{\mathrm{PGU}}
\newcommand{\Inndiag}{\mathrm{Inndiag}}
	
\newcommand{\PSL}{\mathrm{PSL}}
\newcommand{\Om}{\mathrm{\Omega}}
\newcommand{\POm}{\mathrm{P\Omega}}
\newcommand{\PSp}{\mathrm{PSp}}
\newcommand{\PSU}{\mathrm{PSU}}

\makeatletter
\def\imod#1{\allowbreak\mkern7mu({\operator@font mod}\,\,#1)}
\makeatother

    \newtheoremstyle{plain}%
        {8pt plus2pt minus4pt}%
        {8pt plus2pt minus4pt}%
        {\itshape}%
        {}%
        {\bfseries\scshape}%
        {}%
        {1em}%
        {}%

    \newtheoremstyle{upright}%
        {8pt plus2pt minus4pt}%
        {8pt plus2pt minus4pt}%
        {\upshape}%
        {}%
        {\bfseries\scshape}%
        {}%
        {1em}%
        {}%

\theoremstyle{plain}
\newtheorem{thm}{Theorem}[section]
\newtheorem{lem}[thm]{Lemma}
\newtheorem{lemma}[thm]{Lemma}

\newtheorem{cor}[thm]{Corollary}

\theoremstyle{upright}
\newtheorem{remark}[thm]{Remark}

\title[Further analogues of Baer--Suzuki]{Further solvable analogues of the Baer--Suzuki theorem and generation of nonsolvable groups}
\author{Simon Guest}
\begin{document}

\begin{abstract}
Let $G$ be an almost simple group. We prove that if $x \in G$ has prime order $p \ge 5$, then there exists an involution $y$ such that $\langle x,y\rangle$ is not solvable. Also, if $x$ is an involution then there exist three conjugates of $x$ that generate a nonsolvable group, unless $x$ belongs to a short list of exceptions, which are described explicitly. We also prove that if $x$ has order $6$ or $9$, then there exists two conjugates that generate a nonsolvable group. 
\end{abstract}
\maketitle

\section{Introduction}
The following theorem is proved in \cite{Guest1}, and provides a solvable analogue of the classical Baer--Suzuki theorem for elements of certain orders.

\begin{thm} \label{thmA}
Let $G$ be a finite group and suppose that $x$ is an element of prime order $p$ where $p\ge5$. Then $x$ is contained in the solvable radical of $G$ if and only if $\langle x,x^g\rangle$ is solvable for all $g\in G$. In other words, if $x$ is not contained in the solvable radical of $G$ then there exists $g \in G$ such that $\langle x,x^g\rangle$ is not solvable.
\end{thm}
The proof of Theorem \ref{thmA} is by induction, and it is shown that a minimal counterexample to Theorem \ref{thmA} would have to be an almost simple group. Theorem \ref{thmA} is then proved (in \cite{Guest1}) with the following result for almost simple groups. 

\begin{thm} \label{thmA*}
Let $G$ be an almost simple group with socle $G_0$. Let $x \in G$ have odd prime order $p$. Then one of the following holds.
\begin{enumerate}
\item[(1)] There exists $g \in G$ such that $\langle x,x^g\rangle$ is not solvable;
\item[(2)] $p=3$ and $x$ is a long root element in a simple group of Lie type defined over $\mathbb{F}_{3}$, $x$ is a short root element in $G_2(3)$, or  $x$ is a pseudoreflection and $G_0\cong PSU(d,2)$.
\end{enumerate}
\end{thm}

In this paper, we prove a result that is quite similar to Theorem \ref{thmA*}. 

\begin{thm} \label{inv} Suppose that the finite group $G$ satisfies one of the following conditions:
\begin{enumerate}
\item[(1)] $G$ is almost simple group;
\item[(2)] $SL(d,q) \le G \le GL(d,q)$ or $SU(d,q) \le G \le GU(d,q)$, and if $d=2$ and $q$ is odd, then $SL(2,q)$ or $SU(2,q)$ has even index in $G$;
\item[(3)] $G$ is a finite group of Lie type (in the sense of \cite{Steinberg})
 and $G \not \cong SL(2,q)$ ($q$ odd).
\end{enumerate} 
  If $x \in G$ has prime order $p \ge 5$ in  $G/Z(G)$, then there exists an involution $y \in G$ such that
  $\langle x,y\rangle$ is not solvable.
\end{thm}

In particular, Theorem \ref{inv} shows that if  $p\ge 5$ and $G$ is almost simple, then  there exists \emph{an involution} $y$ such that $\langle x,x^y\rangle$ is not solvable. For $\langle x,x^y\rangle$ has index $1$ or $2$ in $\langle x,y\rangle$ and so either both groups are solvable, or both of them are not solvable. Also, Theorem \ref{thmA*} shows that when the order of $x$ has a prime divisor $p \ge 5$ and $G$ is almost simple, there exist two conjugates that generate a nonsolvable group. In this paper we prove an analogous result for elements of order divisible by $3$.


\begin{thm} \label{thm:9} \label{6}
Suppose that $G$ is an almost simple group and that $x$ has order $6$ or $9$. Then there exists an element $g \in G$ such that $\langle x,x^g\rangle$  is not solvable.
\end{thm}

Theorem \ref{thmA*} and Theorem \ref{thm:9} yield the following corollary immediately.
\begin{cor} \label{cor:not2elements}
Let $G$ be an almost simple group with socle $G_0$ and suppose that $x$ in $G$ is not a $2$-element. Then there exists $g$ in $G$ such that $\langle x,x^g
\rangle$ is not solvable or $x$ has order 3 and $x$ is a long root element in a simple group of Lie type defined over $\mathbb{F}_{3}$, a pseudoreflection in $PGU(d,2)$ or a short root element in $G_2(3)$. Moreover, there exist three conjugates of $x$ that generate a nonsolvable group unless $G_0 \cong PSU(d,2)$ or $PSp(d,3)$.
\end{cor}

Guralnick, Flavell, and the author prove in \cite{FGG} that for all nontrivial elements $x$ in a finite (or linear) group $G$, $x$ is contained in the solvable radical of $G$ if and only if any four conjugates of $x$ generate a solvable group. In particular, if $x$ is contained in an almost simple group $G$, then there exist four conjugates of $x$ that generate a nonsolvable group  (this result and  Theorem \ref{thmA} are obtained independently by Gordeev, Grunwald, Kunyavski, and  Plotkin in \cite{GGKP}). Thus if we allow $x$ to be a $2$-element, then a similar result to Corollary \ref{cor:not2elements} is true but with four conjugates of $x$. Corollary \ref{cor:not2elements} and Theorem \ref{thm:3invs} show that in most cases, there exist three conjugates of $x$ that generate a nonsolvable  group.

\begin{thm} \label{thm:3invs}
Let $G$ be an almost simple group with socle $G_0$ and $x$ an involution in $G$.
Then either there exist $g_1,g_2 \in G$ such that $\langle
x,x^{g_1},x^{g_2} \rangle$ is not solvable or $(x,G_0)$
belongs to Table \ref{table:exceptions}.
\end{thm}
\begin{table}[htp]
\centering
\captionsetup{width=0.9\textwidth}
             \caption{\label{table:exceptions} Pairs $(x,G_0)$ such that any three conjugates
      of $x$ in $\Aut(G_0)$ generate a solvable group.}
\begin{tabular}[t]{cc}
          \hline
$G_0$  &$x$ \\
         \hline
         $A_n$ & Transposition \\
         $A_6$  & Triple transposition \\
         $PSU(d,2)$ & Unitary transvection \\
         $PSU(4,2) \cong P\Omega(5,3)$ & Graph automorphism \\
         $PSL(4,2) \cong A_8$ & Graph automorphism \\
         $P\Omega^{\pm}(d,2)$, $d$  even & Orthogonal transvection  \\
         $PSp(d,2) \cong P\Omega(d+1,2)$ & Symplectic transvection
         \\
         $P\Omega(d,3)$, $d$ odd & reflection \\
         $Fi_{22}$  & $x$ in class 2A\\  
         $Fi_{23}$ & $x$ in class 2A\\
         $Fi_{24}^{\prime}$  & $x$ in class 2C in $Fi_{24}^{\prime}:2$\\
\hline
        \end{tabular} 
  \end{table}
  
  We note that if $x$ is an involution, then $\langle x,x^g\rangle$ is dihedral and so we need at least three conjugate involutions to generate a nonsolvable group. 
  
In a future work, the author hopes to improve Corollary \ref{cor:not2elements} to find the minimal number of conjugates in an almost simple group required to generate a nonsolvable group for $2$-elements as well. This requires a  proof  that for an element of order $4$, there exist two conjugates that generate a nonsolvable group with a short list of exceptions, and that two conjugates always suffice for an element of order $8$.

 Also, using Lemma \ref{lem:1.1} below, we get the following corollaries to Theorems \ref{thm:3invs} and  \ref{thm:9}.
 
\begin{cor} \label{cor:to3invs}
 Let $G$ be a finite group with trivial Fitting subgroup and let $x$ be an involution in $G$. Then either there exist elements $g_1,g_2 \in G$ such that $\langle x,x^{g_1},x^{g_2} \rangle$ is not solvable or for every component $L$ of $G$, $x \in N_G(L)$ and $(x,L)$ is in Table \ref{table:exceptions}. 
\end{cor}

\begin{cor} \label{cor:to9}
 Let $G$ be a finite group and let $x \in G$ have order $9$. If $x^3$ is not contained in the solvable radical of $G$ then there exists $g \in G$ such that $\langle x,x^g\rangle$ is not solvable.
\end{cor}
We note that the analogous result to Corollary \ref{cor:to9} for order $6$ elements is not true. For example, let $G= S_5 \times PSL(3,3)$ and $x=(a,b) \in G$ with $a$ a transposition in $S_5$ and $b$ a transvection in $PSL(3,3)$. Then $x$ has order $6$, the solvable radical is trivial, and $\langle x,x^g\rangle$ is solvable for all $g \in G$. We discuss this in more detail in Remark \ref{cor:to6} following the proof of Corollary \ref{cor:to9}.

\section{Preliminaries}
	Throughout the paper, we will use the notation $L^{\epsilon}(d,q)$ where $ \epsilon \in \{\pm\}$ to denote the $PSL(d,q)$ when $\epsilon = +$ and $PSU(d,q)$ when $\epsilon=-$. $D_{n}^{\epsilon}$ will refer to $D_n(q)$ and ${^2}D_n(q)$ for $\epsilon=+$ and $\epsilon=-$ respectively. Similarly $E_6^{\epsilon}(q)$ refers to $E_6(q)$ and ${^2}E_6(q)$ for $\epsilon=+$ and $\epsilon=-$.

Lemma \ref{lem:1.1} below relies on the  result of Guralnick and Kantor \cite{GurAs} that every nontrivial element $x$ in an almost simple group belongs to a pair of elements $(x,y)$ that generates a group containing the socle of $G$. Corollary \ref{cor:to3invs} follows immediately from Theorems \ref{thm:3invs} and Lemma \ref{lem:1.1}.
\begin{lem}\label{lem:1.1}
Let $G$ be a finite group with trivial Fitting subgroup. Let $L$ be a component of $G$ and suppose that $x \not\in N_G(L)$. If $x^2 \not\in C_G(L)$ then there exists $g \in G$ such that $\langle x,x^g\rangle$ is not solvable. In any case, there exist $g_1,g_2 \in G$ such that $\langle x,x^{g_1},x^{g_2} \rangle$ is not solvable.
\end{lem}
\begin{proof} 
See \cite[Lemma 1]{Guest1}.
\end{proof}

\begin{lemma} \label{lem:2.2}
  Let $G_0$ be a simple group of Lie type, let $G 
= Inndiag(G_0)$ and let $x \in G$. 
\begin{enumerate}
\item[(a)] If $x$ is unipotent, let $P_1$ and $P_2$ be distinct maximal parabolic subgroups containing 
a common Borel subgroup of $G$, with unipotent radicals $U_1$ and $U_2$. Then $x$  is 
conjugate to an element of $P_i \backslash U_i$  for $i = 1$ or $i = 2$. 
\item[(b)] If $x$  is semisimple, assume that $x$  lies in a parabolic subgroup of $G$. If the rank of $G_0$  
is at least 2, then there exists a maximal parabolic $P$  with a Levi complement $J$  such 
that $x$  is conjugate to an element of $J$  not centralized by any Levi component (possibly 
solvable) of $J$.
\end{enumerate}
\end{lemma}
\begin{proof} 
 See \cite[Lemma 2.2]{GS} 
\end{proof}

\section{Proof of Theorem 1.3}

Let $(x,G)$ be a minimal counterexample. If $G$ is almost simple, then let $G_0$ be the simple group $G_0$ satisfying
$G_0 \unlhd G \le {\rm Aut}(G_0)$.
\begin{lemma}  \label{lem:Aninvs}
 If $G$ is almost simple and  $G_0 \cong A_n$, then $(x,G)$ is not a minimal counterexample. 
\end{lemma}
\begin{proof} 
  Since $x$ has odd order $p$, it must lie in $A_n$. It suffices to assume that
  $x=(12 \cdots p) \in A_p$. If we let
  $y=(12)(34)$, then $\langle x,y \rangle = A_p$, which is not solvable.
 \end{proof}
 
\begin{lemma} \label{lem:liftinvs}
(a) If $x \in G \le PGL^{\epsilon}(d,q)$ does not lift to an element of order $p$ in $\GL(d,q)$, then $(x,G)$ is not a minimal counterexample.\\
(b)  If $Z(G) \ne \{1\}$, then we may assume that $x \in G$ has order $p$. 
\end{lemma}
\begin{proof} 
To prove (a), note that if $x$ does not lift to an element of order $p$ in $GL^{\epsilon}(n,q)$, then
$p \mid (q- \epsilon,n)$ and the natural $\langle x\rangle$-module $V$ decomposes into $p$-dimensional spaces (see \cite[Lemma 3.11]{Bur2} for example). It therefore suffices to assume that $n=p$ and $x$ acts irreducibly on the natural module $V$ since $(x,G)$ is a minimal counterexample.  Under these conditions on $n$, $p$ and $q$, a Sylow $p$-subgroup of $GL^{\epsilon}(n,q)$ is contained in a type $(q- \epsilon)\wr S_p$ maximal subgroup.  The irreducibility of $x$ implies that $x$ is non-trivial in
$S_p$, and we can take an involution $y \in SL^{\epsilon}(p,q)$ that induces any involution in $S_p$; thus $(x,G)$ cannot be a minimal
counterexample.  \\
To prove (b), if  $SL^{\epsilon}(d,q) \le G \le GL^{\epsilon}(d,q)$, then consider  $x \in G/Z(G) \le PGL(d,q)$. If $x$ does not lift to an element of order $p$ in $G$, then the same argument as for part (a) shows that there exists an involution $y \in SL(d,q)$ such that $\langle x,y\rangle$ is not solvable.  In all other cases, $p$ does not divide $|Z(G)|$ so $x'= x^{|Z(G)|}$ will have order $p$ in $G$ and $(x',G)$ will also be a minimal counterexample to Theorem \ref{inv}.
\end{proof}

\begin{lemma} \label{lem:psl2}  
  If $\PSL(2,q) \le G \le \Aut( \PSL(2,q) )$ or $SL^{\epsilon}(2,q) \lneqq G \le GL^{\epsilon}(2,q)$ with $[G:SL^{\epsilon}(2,q)]$ even, then $(x,G)$ cannot be a minimal counterexample.
\end{lemma}
  \begin{proof} 
  First note that if $\PSL(2,q) \le G \le \Aut(\PSL(2,q) )$, then the order of $x$ implies that $x$ is
  either in $PSL(2,q)$, or it is a field automorphism.  In this case, we may assume that $q \ge 7$ since we have eliminated the case that $A_n \le G \le \Aut(A_n)$. First, let us assume that $x \in PSL(2,q)$. 
  \par
    If $p \mid q$, then $x \in PSL(2,q)$ is a transvection, and we may assume that
  $x=x_{\alpha_1}(a)$, with $a \in \mathbb{F}_p$. In this case,
  $x^{n_{\alpha_1}}=x_{-\alpha_1}(\pm a)$, thus $\langle
    x,x^{n_{\alpha_1}}\rangle = PSL(2,p)$, which is not solvable.
    \par

  If $x$ is semisimple in $PSL(2,q)$, then either $p \mid q+1$, or $p \mid q-1$. Suppose
  first that $p \mid q+1$. Then consider the possibilities for the maximal subgroups
  of $PSL(2,q)$ containing $x$. Since $(x,G)$ is a minimal
  counterexample, $x$ cannot be contained in $A_5$, and it cannot
  be contained in $A_4$ or $S_4$ since $p \ge 5$. Moreover, $x$
  cannot be contained in a subfield subgroup since, because of the
  order of $x$, any such subfield subgroup would be almost simple.  So $x$ can only be
  contained in a dihedral group $D$ of order $\frac{2(q+1)}{(2,q-1)}$. It can be contained in only one
dihedral  subgroup since $C_G(x)$ is the cyclic subgroup of $D$ of order $ \frac{q+1}{(2,q-1)}$. So, let $y$ be an involution in $G$ that is not
  contained in $D$.  \par

  Now suppose that $p \mid q-1$. The possible maximal subgroups
  containing $x$ are a dihedral group $D$ of order $\frac{2(q-1)}{(2,q-1)}$ and (at most two)
  Borel subgroups. Let $i_2(H)$ denote the number of involutions in a
  group $H$. Then
\begin{equation*}
  i_2(G) \ge   \begin{cases}
    q^2-1  &\text{for $q$ even;} \\
    q(q-1)/2 & \text{for $q$ odd.}
       \end{cases}
\end{equation*}
Moreover, if $B$ is a Borel subgroup, then
\begin{equation*}
  i_2(B) \le   \begin{cases}
    q-1 &\text{for $q$ even;} \\
    q  & \text{for $q$ odd.}
  \end{cases}
\end{equation*}
If $D$ is the dihedral group above, then
\begin{equation*}
  i_2(D)  \le \begin{cases}
    (q+1)/2 & \text{for $q$ odd;} \\ 
    q-1 & \text{for $q$ even.}
  \end{cases}
\end{equation*}
So if $q$ is odd, then we may assume that $q \ge 7$; thus
\begin{displaymath}
  i_2(G) \ge (q^2-q)/2 > 2q+(q+1)/2 \ge 2i_2(B)+i_2(D).
\end{displaymath}
Also, if $q$ is even, then we may assume that $q \ge 8$ and so
\begin{displaymath}
  i_2(G) \ge q^2-1 > 2(q-1) + (q-1) \ge 2i_2(B) +i_2(D).
\end{displaymath}
Thus $x \not\in PSL(2,q)$. 

Now suppose that $x$ is a field
automorphism of $PSL(2,q)$. We may assume that $x$ is a standard
field automorphism by \cite[7.2]{GL}. Define $q_0$ by
$q:=q_0^p$ and let
\begin{displaymath}
  \Gamma = \{y \in G_0 \mid y^2=1, \: \langle  x,y\rangle \ne  G
  \}.
\end{displaymath}
We will show that $|\Gamma| < i_2(G_0)$. Indeed, if $y \in \Gamma$
then $\langle x,y\rangle \cap G_0$ is contained in a
subgroup of $G_0=PSL(2,q_0^p)$. From the description of the subgroups of $PSL(2,q)$, since $p$ is odd, $\langle x,y \rangle \cap G_0$ must be contained in a Borel subgroup, a
dihedral group of order $\frac{2(q \pm 1)}{(2,q-1)}$, or a subfield subgroup of type $PSL(2,q_0)$. 
We note
that since $(x,G)$ is a minimal counterexample, $\langle x,y \rangle \cap G_0$ cannot be contained in any
other \emph{maximal} subfield subgroups. Now, if $H$ is a torus of order $\frac{(q \pm 1)}{(2,q-1)}$, a Borel or subfield subgroup,
then the $G$-conjugates of $H$ fixed by $x$ form one $C_{G_0}(x)$
orbit (see the proof of \cite[Lemma 3.1]{GS} for example). If $H$ is a $G$-conjugate of a maximal dihedral group that is fixed by $x$, then $x$ must also normalize the characteristic cyclic subgroup of $H$ (a torus of order $\frac{(q \pm 1)}{(2,q-1)}$). Since the $G$-conjugates of the torus that are fixed by $x$ are all $C_{G_0}(x)$-conjugate, it follows that the $G$-conjugates of the dihedral group that are fixed by $x$ are also $C_{G_0}(x)$-conjugate. 
So the number of conjugates of $H$ that can contain $\langle x,y \rangle \cap G_0$ is at most  
$|C_{G_0}(x)|/|C_H(x)|$. Thus the number of involutions $y$ in
$G_0$ such that $\langle x,y\rangle \cap G_0$ is contained in a
conjugate of $H$ is at most
\begin{displaymath}
  \frac{i_2(H)|C_{G_0}(x)|}{|C_H(x)|}.
\end{displaymath}
Let $X_1, \ldots, X_k=C_{G_0}(x)$ be representatives for the
conjugacy classes of maximal subgroups containing $\langle
x,y\rangle \cap G_0$. Note that there are no nontrivial conjugates of $X_k=C_{G_0}(x)$ fixed by $x$ and so a crude upper bound for the number of involutions in $G$ such that $ \langle x,y \rangle \cap G_0$ is contained in $C_{G_0}(x)$ is $|C_{G_0}(x)|$.  So if $(x,G)$ is a minimal counterexample, then we have
\begin{displaymath}
  |\Gamma| \le \sum_{i=1}^{k-1}
  \frac{i_2(X_i)|C_{G_0}(x)|}{|C_{X_i}(x)|}+|C_{G_0}(x)|.
\end{displaymath}
If $q$ is odd, then
 \begin{align*}
    \sum_{i=1}^k \frac{i_2(X_i)|C_{G_0}(x)|}{|C_{X_i}(x)|} \le&
    \frac{q_0^pq_0(q_0^2-1)}{q_0(q_0-1)} + \frac{(q_0^p+1)q_0(q_0^2-1)}{2(q_0-1)} \\
    & +\frac{(q_0^p+3)q_0(q_0^2-1)}{2(q_0+1)} + \frac{q_0(q_0^2-1)}{2} \\
    \le &\frac{q_0(q_0+1)(3q_0^p+q_0+3)}{2};
  \end{align*}
but this is less than $i_2(G_0) \ge q_0^p(q_0^p-1)/2$. If $q$ is
even, then
 \begin{align*} 
   \sum_{i=1}^k \frac{i_2(X_i)|C_{G_0}(x)|}{|C_{X_i}(x)|} \le&
   \frac{(q_0^p-1)q_0(q_0^2-1)}{q_0(q_0-1)} +
   \frac{(q_0^p-1)q_0(q_0^2-1)}{2(q_0-1)} \\
   & +\frac{(q_0^p+1)q_0(q_0^2-1)}{2(q_0+1)} + q_0(q_0^2-1) \\
   \le & 2(q_0^p+q_0)(q_0+1)q_0;
\end{align*}
but $i_2(G_0) \ge (q_0^{2p}-1)$  and so $|\Gamma| \le i_2(G_0)$. 
\par
If $SL^{\epsilon}(2,q) \lneqq G \le GL^{\epsilon}(2,q)$ and $x \in G$, then we may assume that $x$ has order $p$ by Lemma \ref{lem:liftinvs}. Moreover, we may assume that $q$ is odd since $PSL(2,2^a) \cong SL(2,2^a)$ and so our hypothesis states that $[G:SL(2,q)]$ is even. If $x$ is semisimple, then since $SL^{\epsilon}(2,4)\cong PSL(2,4)$, $GL^{\epsilon}(2,4)$ is not a minimal counterexample and $GL^{\epsilon}(2,5)$ does not contain semisimple elements of order $p \ge 5$; thus we may assume that $q \ge 7$. If $(x,G)$ is a minimal counterexample then $x$ must be contained in $SL^{\epsilon}(2,q)$, for otherwise $p | q- \epsilon$ and there exists a scalar $\lambda$ such that $\lambda x \in SL^{\epsilon}(2,q)$. 
Thus $SL^{\epsilon}(2,q)$ has index $2$ in $G$, and there are at least $q^2+q$ involutions in $G$. Now the same counting argument as for $PSL(2,q)$ shows that $(x,G)$ cannot be a minimal counterexample.  

If $x$ is unipotent in $SL^{\epsilon}(2,q)$, then $q \ge 5$, and by minimality, $SL^{\epsilon}(2,q)$ has index $2$ in $G$. We may assume that $x$ is not contained in any subfield subgroups by minimality. So choose an involution $y$ such that $[x,x^y] \ne 1$. 
 Another inspection of the maximal subgroups shows that $\langle x,y\rangle$ is not solvable. 
\end{proof}

\begin{lem} \label{lem:outinvs}
  If $G$ is almost simple and $x$ is an outer automorphism of $G_0$, then $(x,G)$ cannot be a
  minimal counterexample, except possibly if $G_0 \cong {^2}B_2(2^a)$.
\end{lem}
 \begin{proof}
   We may assume that the untwisted Lie rank is at least 2 since the case where
   $G_0 \cong PSL(2,q)$ has already been eliminated.  Since $x$ has order $p$, it is a field automorphism, and by \cite[7.2]{GL} we may assume that $x$ is
   a standard field automorphism. 
   Now if $G_0$ is not a Suzuki--Ree group, then $x$ normalizes but does not centralize an $SL(2,q)$ subgroup $S$. So if $q$ is even and $G_0$ is not a Suzuki--Ree group, then there exists $y$ an involution $y \in S$ such that $\langle x,y\rangle$ is not solvable. Thus we may assume that either $G_0$ is a Suzuki--Ree group or that $q$ is odd. 
   
   If $q$ is odd,  then an inspection of the (extended) Dynkin diagram shows that $x$ normalizes but doesn't centralize a type $SL(3,q)$ subgroup $H$, unless $G_0 \cong PSL(2,q)$, $PSL(3,q)$, $PSp(4,q)$, ${^3}D_4(q)$, ${^2}G_2(3^a)$, or $PSU(d,q)$. 
   If $G_0 = L^{\epsilon}(3,q)$ and $q$ is odd, 
   then $x$ normalizes a subgroup of type $SO(3,q)$. 
   If $G_0 \cong PSU(d,q)$ and $d \ge 4$, then $x$ normalizes but does not centralize a subgroup H of $G_0$ that is isomorphic to $PSO^{\epsilon}(d,q).(d,2)$ (when $d=4$, take $\epsilon=-$) by \cite[4.5.5]{KL}. \\
    If $G_0 \cong {^3}D_4(q)$ then $x$ normalizes but does not centralize a subgroup $H$ isomorphic to $G_2(q)$. If $G_0 \cong PSp(4,q)$, then $x$ normalizes a subgroup $H$ isomorphic to $PSp(2,q^2).2$ (see \cite[Propostion 4.3.10]{KL}). 
 If $G_0\cong {^2}G_2(3^{a})$, then let $z$ be an involution in $C_{G_0}(x)$. Then  $x \in C_G(z)$, which is a subgroup $H$ of type $PSL(2,3^{2a})$ by \cite[Table 4.5.1]{GLS}. 
  Moreover, $x$ does not centralize a subgroup of type $PSL(2,3^{2a})$ since it doesn't centralize an element of order divisible by $3^{2a}+1$.  
    If $G_0 \cong {^2}F_4(2^a)$, then a field automorphism normalizes, but does not centralize,  a subgroup of $G_0 \cong {^2}F_4(2^a)$ isomorphic to $PGU(3,2^a):2$ by \cite{2F4}.  
    By minimality,  it follows in all cases that there exists an involution $y \in H$ such that $\langle x,y\rangle$ is not solvable. 
   \end{proof}

 \begin{lem} \label{unipotents} 
 If $x$ is a unipotent element in $G$,  then $(x,G)$ cannot be a minimal counterexample.
\end{lem}
\begin{proof}
Since $p \ge 5$ and $p|q$, $G$ cannot be a Suzuki--Ree group, and by Lemma \ref{lem:psl2}, we may assume that the untwisted Lie rank is at least $2$. 
If $G$ is an almost simple group, then we may assume that $G=G_0$ and by Lemma \ref{lem:liftinvs}, we can lift $x$ to an element of order $p$ in the universal version of $G_0$. 
By \cite[Lemma 2.1]{GS}, 
 we may assume that $x$ is nontrivial in $P/U$ for some end node maximal parabolic subgroup $P$, with unipotent radical $U$, unless $G$ is ${^3}D_4(q)$, or $SU(d,q)$. 
 
So we may assume that $x$ acts nontrivially on a Levi subgroup $L$, and since $G$ is simply connected, so is $L$ (see \cite[2.6.5(f)]{GLS} for example). By induction, there exists an involution $y \in L$ such that $\langle x,y\rangle$ is not solvable; thus there exists an involution $y' \in G_0$ such that $\langle x,y\rangle$ is not solvable.   

If $G_0 \cong {^3}D_4(q)$, then we may assume that $x$ is nontrivial in $P/U$, for a maximal parabolic subgroup $P$. The Levi complement is of type $SL(2,q)$ or $SL(2,q^3)$, but a split torus normalizes both of these Levi complements and induces diagonal automorphisms on them. Thus we can reduce to the case that $SL(2,q) \le G \le GL(2,q)$, where $SL(2,q)$ has even index in $G$ when $q$ is odd.

 Now suppose that $G = SU(d,q)$ and $d \ge 5$. Then \cite[Lemma 2.1]{GS} implies that we may assume that $x$ is nontrivial in $P/U$, for some (not necessarily end-node) maximal parabolic subgroup $P$.  Therefore $x$ will act nontrivially on one of the components of the Levi complement of $P$, and these components are all nonsolvable since $p \ge 5$, and not of type $SL(2,q)$. 
 
 If $G = SU(4,q)$, then we may assume that $x$ is nontrivial in $P/U$ for some maximal parabolic subgroup $P$. Now the Levi complement $L$ of $P$ in $SU(4,q)$ is either isomorphic to $GU(2,q)$ or a normal subgroup of $GL(2,q^2)$ of index $q-1$; thus $(x,G)$ cannot be a minimal counterexample.
 
  The only other possibility is $G=SU(3,q)$. If $x$ is a transvection, then it is contained in a subgroup isomorphic to $GU(2,q)$. So
we may assume that $x$ is not a transvection and 
$x$ is therefore regular unipotent. Since all inner diagonal involutions of $PSU(3,q)$ lift to involutions in $SU(3,q)$, we can work in $PSU(3,q)$. From the list of maximal subgroups of $PSU(3,q)$ (see \cite[Theorem 6.5.3]{GLS} for example), we may assume that the only maximal subgroups that could contain $x$ are the maximal
parabolic subgroups since the other maximal subgroups of order divisible by $p$ are almost simple. 
Now $x$  only stabilizes one totally singular $1$-space, and so is only contained in one maximal parabolic subgroup. 
 So choose an involution $y$ that is not contained in this maximal parabolic subgroup. Then $\langle
x,y \rangle$ is not solvable. 
\end{proof}


 \begin{lemma} 
 If $G$ or $G_0$ is a classical group, then $(x,G)$ cannot be a minimal counterexample. 
\end{lemma}
 \begin{proof} 
  By Lemmas \ref{lem:psl2}, \ref{lem:outinvs}, \ref{lem:liftinvs} and \ref{unipotents} we may assume that $x$ is semisimple and that $G_0$ or $G/Z(G)$ is not $PSL(2,q)$. Moreover, we can and will assume that $x$ is an element of order $p$ in $G$ where $SL(n,q) \le G \le \GL_{n}(q)$, $SU(d,q) \le G \le \GU(n,q)$, $G=Sp(n,q)$ or $G=\Om^{\epsilon}(n,q)$ by Lemma \ref{lem:liftinvs} and \cite[Lemma 3.11]{Bur2}. In case O, we may assume that $d \ge 7$. If $G$ is a unitary group, let $e$ be the smallest positive integer such that $p \mid q^{2e}-1$; otherwise let  $e$ be the smallest positive integer such that $p \mid q^{e}-1$. Consider a decomposition of $V$ into irreducible $\langle x\rangle$-invariant spaces
  \begin{align} \label{V}
 V= (W_1 \oplus W_1') \perp \cdots \perp  (W_k \oplus W_k') \perp U_1 \perp \cdots \perp U_l
\end{align}
  where the $W_i$ and $W_i'$ are totally singular, and the $U_i$ and $W_i \oplus W_i'$ are nondegenerate. Each irreducible subspace on which $x$ acts nontrivially  has dimension $e$. In case $U$, we can and will assume that the $1$-spaces on which $x$ acts trivially are nondegenerate. We consider five cases separately.
 
  \begin{list}{\labelitemi}{\leftmargin=0.5em}
 \item[(i)] Suppose that $e=1$. In cases L, S, and O, all of the irreducible subspaces on which $\langle x\rangle$ acts nontrivially must be totally singular since $p | q-1$. Moreover, $q \ge 7$ since $p \ge 5$. So in cases S and L, we may assume that 
$x$ acts nontrivially on $W_1 \oplus W_2$ and so $x$ is contained in a type subgroup $\GL(2,q)$, and  $(x,G)$ is not a minimal counterexample in this case. In case O, since $d \ge 7$, we may assume that there are totally singular subspaces $W_1$, $W_2$, $W_3$ such that $W_1 \oplus W_2 \oplus W_3$ is totally singular and $x$ invariant; thus $x$ is contained in a type $SL(3,q)$ subgroup. In case U, if $p|q-1$, then we can argue as in cases S and L to reduce to the case $G \cong GL(2,q^2)$.  If  $p | q+1$, then $q \ge 4$ and we may assume that all of the subspaces in (\ref{V}) are nondegenerate; 
 so $x$ is contained in a type $GU(1,q)^d$ subgroup and therefore we can reduce to the case $G\cong GU(2,q)$. 
\item[(ii)] Suppose that $e=2$. In case $U$, all of the $2$-spaces in (\ref{V}) are totally singular since even dimensional unitary groups do not contain irreducible elements. So in case $U$, $x$ acts irreducibly on $W_1$ and we reduce to the case $G \cong GL(2,q)$. In the other cases, $q\ge 4$ since $p \ge 5$. In cases L and S, if there is a totally singular $2$-space $W_1$ in (\ref{V}), then we can reduce to the case $G \cong GL(2,q)$.  If there are no totally singular $2$-spaces in case S, then all of the $2$-spaces in \ref{V} are nondegenerate,  and we can reduce to the case $G \cong Sp(4,q)$. In this case, we may assume that $q$ is odd since  if $q$ is even then we can reduce to the case $G \cong Sp(2,q)$. But when $q$ is odd, a Sylow $p$-subgroup is contained in a subgroup isomorphic to $GU(2,q)$ (see \cite[p. 118]{KL}); thus we do not have a minimal counterexample in this case either. 
In case O, either we can reduce to the case $G \cong \Om(d,q)$ with $d=5$ or $6$,  or all of the subspaces are totally singular. In this case, $x$ stabilizes $W_1 \oplus W_2$, and is thus contained in a subgroup of type $SL(4,q)$. Thus $(x,G)$ cannot be a minimal counterexample.
\item[(iii)] Suppose that $e=3<d$. If there is a totally singular $3$-space $W_1$ in (\ref{V}), then in all cases $x$ will be contained in a subgroup of type $SL(3,q)$ (or $SL(3,q^2)$). Otherwise, all of the $3$-spaces in \ref{V} are nondegenerate and we are in case U or O. 
In case U, we can reduce to the case $G \cong SU(3,q)$, and $q \ne 2$ since $GU(3,2)$ has order $2^3 3^4$. In case O, we have $d \ge 7$ and so we can reduce to the case $G = \Om^{\pm}(6,q)$. 
\item[(iv)] Suppose that $4 \le e < d$. If there is a totally singular $e$-space in (\ref{V}), then we can reduce to the $e$ dimensional linear case. Otherwise $x$ acts irreducibly on a nondegenerate $e$-space, and we can reduce to the case $G= \Sp(e,q)$ in case S, $SU(e,q) \le G  \le GU(e,q)$ in case U ($e$ odd), and $G = \Om^{-}(e,q)$ ($e$ even) in case O.  
\item[(v)] Suppose that $e=d$, so that $x$ acts irreducibly. In case S, $x$ must be contained in a  $GU(d/2,q)$ subgroup  \cite[p. 118]{KL}. In case O, if $d/2$ is odd then $x$ must be contained in a subgroup $H$ of type $GU(d/2,q)$. 
 If $d/2$ is even, then $H$ is contained in a  $\Om^{-}(d/2,q^2)$ subgroup . So we may assume that $G$ is linear or unitary. Now observe that if $d$ is even, then $G$ is linear and   $x$ is contained in a normal subgroup of $GL(2,q^{d/2})$ of index dividing $q-1$ (\cite[(4.3.16)]{KL}), and so if $(x,G)$ is a minimal counterexample, then $d/2$ must be odd. But if $d/2$ is odd, then $x$ is contained in a type $GL(d/2,q^2)$ subgroup and so $(x,G)$ cannot be a minimal counterexample in this case either. So $d$ must be odd and in fact $d$ must be an odd prime since otherwise $x$ is contained in a type $GL^{\epsilon}(d/r,q^r)$ subgroup. We can list the possible maximal subgroups of $G$ that could contain $x$ using \cite{GPPS} and \cite{KL}. 
Since $d$ is odd, all involutions in $PGL^{\epsilon}(d,q)$ lift to involutions in $SL^{\epsilon}(d,q)$; 
 thus we can work in the almost simple group $G/Z(G)$. In particular, we may assume that $x$ is not contained in any almost simple  subgroup of $G/Z(G)$.  In this case, the only possible maximal subgroups containing $x$ are of type $GL^{\epsilon}(1,q^d).d$. Since $p \nmid d$, $x$ is contained in a cyclic maximal torus $T$, and since $C_G(x)=T$, $x$ is contained in only one maximal subgroup. Thus we can pick an involution $y$ not contained in this maximal subgroup, and $\langle x,y\rangle$ will not be solvable. \qedhere
 \end{list}

\end{proof}

\begin{lemma} \label{lem:nosc}
 If $G$ is a finite group of Lie type and $Z(G) \ne 1$, then $(x,G)$ cannot be a minimal counterexample unless $G$ is (simply connected) $E_7(q)$.
 \end{lemma}
\begin{proof} 
 By our previous work, we may assume that $G$ is an exceptional group. But 
 then the centre of $G$ is either trivial or of odd order, or $G$ is $E_7(q)$. So if $(x,G)$ is a minimal counterexample, then Theorem \ref{inv} holds for $G/Z(G)$. But then there exists $y \in G$ such that $\langle x,y \rangle$ is not solvable and  $ y^2 \in Z(G)$. Since $Z(G)$ has odd order, $y'= y^{|Z(G)|}$ is an involution in $G$ and $\langle x,y'\rangle$ is not solvable.
 \end{proof}

\begin{lem} \label{lem:El}
 Suppose that $G$ is almost simple or a finite group of Lie type, and that $G_0$ or $G/Z(G)$ is one of the simple groups $F_4(q)$, $E_6(q)$, $E_7(q)$, $E_8(q)$, or ${^2}E_6(q)$. Then $(x,G)$ cannot be a minimal counterexample. 
\end{lem}
\begin{proof} 
 We may assume that $G$ is almost simple with $x \in {\rm Inndiag}(G_0)$ or that $G$ is (simply connected) $E_7(q)$ by Lemmas \ref{lem:outinvs} and \ref{lem:nosc}. Moreover, we may assume that $x$  is semisimple in both cases by Lemma
\ref{unipotents}. First suppose that $G$ is one of the untwisted groups. If $p|q-1$ then $x$ is contained in a Borel subgroup and therefore in a $P_1$ parabolic subgroup. By Lemma \ref{lem:2.2} we may assume that $x$ is contained in a Levi subgroup of type $C_3(q)$ or $D_{l-1}(q)$,
 and so $(x,G)$ cannot be a minimal counterexample. So we may assume that $p\nmid q-1$ in the untwisted cases.  Now suppose that $x$ is contained in some maximal parabolic subgroup. Again we may assume that $x$ acts noncentrally on each component of the Levi complement. It is easily verified that $(x,G)$ cannot be a minimal counterexample since we can work in one of the groups of Lie type in the Levi complement to find an involution $y$ such that $\langle x,y\rangle$ is not solvable. 
 
So we may assume that $x$ is not contained in any parabolic subgroups. If this is the case, then the centralizer of $x$ is reductive and contains no unipotent
elements. 

First suppose that $G=E_7(q)$, and without loss of generality we may assume that $G$ is simply connected.  We know that
\begin{displaymath} 
	 |G|  = q^{63} \prod_{d_i \in \{2,6,8,10,12,14,18\}} (q^{d_i}-1),
\end{displaymath}
so let $e$ be the smallest $d_i$ such that $p | (q^{d_i}-1)$. If $e=14$ then either $p|q^7-1$ or $p| q^7+1$. If $p|q^7-1$ then the $p$-part of $|G|$ (that is, the largest power of $p$ dividing $|G|$) is the $p$-part of  $(q^7-1)/(q-1)$ and so a Sylow $p$-subgroup is contained in a type $SL(7,q)$ subsystem subgroup. Thus $(x,G)$ cannot be a minimal counterexample in this case. If $p| q^7+1$, then the $p$-part of $|G|$ is the $p$-part of $(q^7+1)/(q+1)$ and so a Sylow $p$-subgroup is contained in a type $SU(7,q)$ subgroup. Similarly we can show that $(x,G)$ cannot be minimal counterexample for all values of $p$. We illustrate this work in Table \ref{tab:SylE7} below (note that $p\nmid q-1$ since we are assuming that $x$ is not contained in any parabolic subgroups). We do the same for $F_4(q)$, $E_6(q)$, ${^2}E_6(q)$, and $E_8(q)$ and record our results in Tables \ref{tab:SylF4}, \ref{tab:SylE6}, \ref{tab:Syl2E6} and \ref{tab:SylE8} respectively. 

The only case where we have not shown that $(x,G)$ is not a minimal counterexample is when $G = E_8(q)$ and $p$ is a primitive prime divisor of $q^{30}-1$ or $q^{15}-1$. It follows that $p  \equiv  1 \imod{15}$ or $p  \equiv  1 \imod{30}$, and in particular, that $p \ge 31$. If $p=31$, then the Sylow $31$-subgroups are cyclic 
and $x$ is contained in an exotic local subgroup $5^3.SL(3,5)$. 
Therefore we may assume that $p \ne 31$. The maximal subgroups of $E_8(q)$ are described in \cite[Theorem 8]{LSe1} and if $(x,G)$ is a minimal counterexample, then $x$ can only be contained in a (single) torus $T$ of type $q^8+q^7-q^5-q^4-q^3+q+1$ or $q^8-q^7+q^5-q^4+q^3-q+1$. In this situation, we can pick an involution $y \in G$ that is not  contained in the normalizer of $T$ and $\langle x,y\rangle$ will not be solvable. \qedhere
\begin{table}[htdp]
\begin{center}
\caption{Subgroups of $E_7(q)$ containing a Sylow $p$-subgroup\label{tab:SylE7}}

\begin{tabular}{cccc}
\hline 
$e$ & $p$ divides & $p$-part of $G$ & Subgroup type containing a  \\
 &  & is $p$-part of & Sylow $p$ subgroup \\
 \hline
 $18$ & $q^9+1$ & $q^6-q^3+1$ & ${^2}E_6(q)$ \\ 
$18$ & $q^9-1$ & $q^6+q^3+1$ & $E_6(q)$ \\
$14$ & $q^7+1$ & $(q^7+1)/(q+1)$ & $SU(7,q)$ \\
$14$ & $q^7-1$ & $(q^7-1)/(q-1)$ & $SL(7,q)$ \\
$12$ & $q^6+1$ & $(q^6+1)/(q^2+1)$ & $F_4(q)$ \\
$12$ & $q^6-1$ & x &  \\
$10$ & $q^5+1$ & $(q^5+1)/(q+1)$ & $SU(7,q)$ \\
$10$ & $q^5-1$ & $(q^5-1)/(q-1)$ & $SL(7,q)$ \\
$8$ & $q^4+1$ & $q^4+1$ & $SL(8,q)$ \\
$8$ & $q^4-1$ & x& \\
$6$ & $q^3+1$ & $(q^2-q+1)^3$ & ${^2}E_6(q)$ \\
$6$ & $q^3-1$ & $(q^2+q+1)^3$ & $E_6(q)$ \\
$2$ & $q+1$ & $(7,p)(5,p)(q+1)^7$ & $SU(8,q)$ \\ 
$2$ & $q-1$ & x &\\ 
\hline
\end{tabular} 
\end{center}
\end{table}

\begin{table}[htdp]
\begin{center}
\caption{ Subgroups of $F_4(q)$ containing a Sylow $p$-subgroup\label{tab:SylF4}}
\begin{tabular}{cccc}
\hline 
$e$ & $p$ divides & $p$-part of $G$ & Subgroup type containing a  \\
 &  & is $p$-part of & Sylow $p$ subgroup \\
 \hline
$12$ & $q^6+1$ & $q^4-q^2+1$ & ${^3}D_4(q)$ \\
$12$ & $q^6-1$ & x &\\
$8$ & $q^4+1$ & $q^4+1$ & $SO(9,q)$ \\
$8$ & $q^4-1$ & $(q^2+1)^2$ & $SO(9,q)$ \\
$6$ & $q^3+1$ & $(q^2-q+1)^2$ & ${^3}D_4(q)$  \\
$6$ & $q^3-1$ & $(q^2+q+1)^2$ & ${^3}D_4(q)$ \\
$2$ & $q+1$ & $(q+1)^4$ & $SO(9,q)$ \\
$2$ & $q-1$ & x &  \\
\hline
\end{tabular} 
\end{center}
\end{table}

\begin{table}[htdp]
\begin{center}
\caption{ Subgroups of $E_6(q)$ containing a Sylow $p$-subgroup\label{tab:SylE6}}
\begin{tabular}{cccc}
\hline 
$e$ & $p$ divides & $p$-part of $G$ & Subgroup type containing a  \\
 &  & is $p$-part of & Sylow $p$ subgroup \\
 \hline
$12$ & $q^6+1$ & $q^4-q^2+1$ & $F_4(q)$ \\
$12$ & $q^6-1$ & x &  \\
$9$ & $q^9-1$ & $q^6+q^3+1$ & $SL(3,q^3)$ \\
$8$ & $q^4+1$ & $q^4+1$ & $F_4(q)$ \\
$8$ & $q^4-1$ & $(q^2+1)^2$ & $F_4(q)$ \\
$6$ & $q^3+1$ & $(q^2-q+1)^2$ & $F_4(q)$ \\
$6$ & $q^3-1$ & $(q^2\!+\!q\!+\!1)^3$ & $SL(3,q) \! \circ \! SL(3,q)\! \circ \!SL(3,q)$ \\
$5$ & $q^5-1$ & $q^4+q^3+q^2+q+1$ & $SL(5,q)$ \\ $2$ & $q+1$ & $(q+1)^4$ & $F_4(q)$\\ 
$2$ & $q-1$ & x &\\ 
\hline
\end{tabular} 
\end{center}
\end{table}

\begin{table}[htdp]
\begin{center}
\caption{ Subgroups of ${^2}E_6(q)$ containing a Sylow $p$-subgroup\label{tab:Syl2E6}}
\begin{tabular}{cccc}
\hline 
$e$ & $p$ divides & $p$-part of $G$ & Subgroup type containing a  \\
 &  & is $p$-part of & Sylow $p$ subgroup \\
 \hline
$12$ & $q^6+1$ & $q^4-q^2+1$ & $F_4(q)$ \\
$12$ & $q^6-1$ & x &  \\
$9$ & $q^9+1$ & $q^6-q^3+1$ & $SU(3,q^3)$ \\
$8$ & $q^4+1$ & $q^4+1$ & $F_4(q)$ \\
$8$ & $q^4-1$ & $(q^2+1)^2$ & $F_4(q)$ \\
$6$ & $q^3+1$ & $(q^2-q+1)^3$ & $SU(3,q) \circ SU(3,q) \circ SU(3,q)$ \\
$6$ & $q^3-1$ & $(q^2+q+1)^2$ &  $F_4(q)$\\
$5$ & $q^5+1$ & $q^4\!-\!q^3\!+\!q^2\!-\!q\!+\!1$ & $SO^{-}(10,q)$ \\ 
$2$ & $q+1$ &x &\\ 
$2$ & $q-1$ &  $(q-1)^4$& $F_4(q)$\\ 
\hline
\end{tabular} 
\end{center}
\end{table}

\begin{table}[htdp]
\begin{center}
\caption{Subgroups of $E_8(q)$ containing a Sylow $p$-subgroup\label{tab:SylE8}}
\begin{tabular}{cccc}
\hline 
$e$ & $p$ divides & $p$-part of $G$ & Subgroup type containing a  \\
 &  & is $p$-part of & Sylow $p$ subgroup \\
 \hline
$30$ & $q^{15}+1$ & $q^8 \!+ \!q^7\! - \! q^5\! - \! q^4\! - \! q^3\! + \! q\! + \! 1$ & see Lemma \\ 
$30$ & $q^{15}-  1$ & $q^8\! - \! q^7\! + \! q^5\! - \! q^4\! + \! q^3\! - \! q\! + \! 1$ & see Lemma\\ 
$24$ & $q^{12}+1$ & $q^8-q^4+1$ & $SU(3,q^4)$ \\ 
$24$ & $q^{12}-1$ & x & \\
$20$ & $q^{10}+1$ & $q^8-q^6+q^4-q^2+1$ & $SU(5,q^2)$ \\ 
$20$ & $q^{5}+1$ & $(q^4-q^3+q^2-q+1)^2$ & $SU(5,q) \circ SU(5,q)$ \\
$20$ & $q^{5}-1$ & $(q^4+q^3+q^2+q+1)^2$ & $SL(5,q) \circ SL(5,q)$ \\
 $18$ & $q^9+1$ & $q^6-q^3+1$ & $SU(9,q)$ \\ 
$18$ & $q^9-1$ & $q^6+q^3+1$ & $SL(9,q)$ \\
$14$ & $q^7+1$ & $q^6\! - \! q^5\! + \! q^4\! - \! q^3\! + \! q^2\! - \! q\! + \! 1$ & $SU(9,q)$ \\
$14$ & $q^7-1$ & $q^6\! + \! q^5\! + \! q^4\! + \! q^3\! + \! q^2\! + \! q\! + \! 1$ & $SL(9,q)$ \\
$12$ & $q^6+1$ & $(q^4-q^2+1)^2$ & $SU(3,q^2)\circ SU(3,q^2)$ \\
$12$ & $q^3+1$ & $(q^2-q+1)^4(5,p)$ & ${^3}D_4(q) \circ {^3}D_4(q)$ \\
$12$ & $q^3-1$ & $(q^2+q+1)^4(5,p)$ & ${^3}D_4(q) \circ {^3}D_4(q)$\\
$8$ & $q^4+1$ & $(q^4+1)^2$ & $SU(3,q^4)$ \\
$8$ & $q^2+1$ &  $(q^2+1)^4(5,p)$& $SU(5,q^2)$\\
$2$ & $q+1$ & $(7,p)(5,p)^2(q+1)^8$ & $SU(5,q)\!\circ \! SU(5,q)$ or $SU(9,q)$\\ 
$2$ & $q-1$ & x &\\ 
\hline
\end{tabular} 
\end{center}
\end{table}
 \end{proof}

\begin{lemma} 
  If $G_0 \cong G_2(q)$, ${^3}D_4(q)$, or $^2F_4(2^{a})$, then $(x,G)$ cannot be a minimal counterexample.
\end{lemma}
\begin{proof} 
The proof is similar to that of Lemma \ref{lem:El}.
We may assume that $x$ is semisimple by Lemmas \ref{unipotents} and \ref{lem:outinvs}. Also, since $G_2(2)' \cong PSU(3,3)$, we can eliminate this case. If $G_0 \cong G_2(q)$, then $x$ normalizes but does not centralize a subgroup of type $SL^{\epsilon}(3,q)$ (see \cite[p. 546]{GS}). So  $G_0 \not\cong G_2(q)$. If $G_0$ is  ${^3}D_4(q)$, or ${^2}F_4(2^{2a+1})$ then we list the possible expressions in $q$ that could be divisible by $p$ in Tables \ref{tab:Syl3D4} and \ref{tab:Syl2F4}. Since $p\ge 5$, $p$ divides precisely one of these expressions. In most cases, we can  deduce that $(x,G)$ cannot be a minimal counterexample. If $G_0 \cong {^2}F_4(2^{a})$, then we may therefore assume that $p | q^4-q^2+1$. In this case, either $p | q^2+\sqrt{2q^3} +q + \sqrt{2q}+ 1$ or 
$p | q^2-\sqrt{2q^3} +q -\sqrt{2q}+ 1$, and from the list of maximal subgroups of $G$ (see \cite{2F4}), we may assume that $x$ is only contained in a (single) torus $T$ of order $q^2+\sqrt{2q^3} +q + \sqrt{2q}+ 1$ or $q^2-\sqrt{2q^3} +q -\sqrt{2q}+ 1$. Thus we can pick an involution $y \in G$ that is not contained in the normalizer of $T$ and $\langle x,y\rangle$ will not be solvable. 

Suppose that $G_0 \cong {^3}D_4(q)$. We note that if $p | q^2 - q+1$, then $x$ is contained in a subgroup of type $(q^2-q+1) \circ SU(3,q)$.  If $x$ does not centralize the $SU(3,q)$, then $(x,G)$ cannot be a minimal counterexample. But if $x$ does centralize the $SU(3,q)$ subgroup, then $x$ centralizes unipotent elements and is therefore contained in a parabolic subgroup. By Lemma \ref{lem:2.2} we may assume that $x$ is noncentral in the Levi subgroup, which is of type $SL(2,q)$ or $SL(2,q^3)$. But if $q > 3$, then $x$ cannot be a minimal counterexample since these Levi components are normalized by a split torus, which induces diagonal automorphisms,  and we can therefore reduce to the case $SL(2,q) \le G \le GL(2,q)$ where $SL(2,q)$ has even index in $G$ when $q$ is odd. We can verify the cases $q=2$ and $q=3$ in MAGMA. 
The case where $p | q^2+q+1$ is the same argument.  The only remaining case is where $p | q^4-q^2+1$.  In this case, the list of maximal subgroups in \cite{3D4} allows us to assume that $x$ is only contained a (single) torus $T$ of order $(q^4-q^2+1)$. We can then choose an involution $y$ that  is not contained in the normalizer of $T$. \qedhere 

\begin{table}[htdp]
\begin{center}
\caption{ Subgroups of ${^2}F_4(q)$ containing a Sylow $p$-subgroup, $q=2^a$\label{tab:Syl2F4}}
\begin{tabular}{ccc}
\hline 
$p$ divides & $p$-part of $G$ & Subgroup type containing a  \\
   & is $p$-part of & Sylow $p$ subgroup \\
 \hline
 $q^4-q^2+1$& $q^4-q^2+1$ & see Lemma\\
$q^2+1$ & $(q^2+1)^2$ & ${^2}B_2(q) \!\circ\! {^2}B_2(q)$\\
$q+1$ & $(q+1)^2$ & $SU(3,q)$ \\
$q^2-q+1$ & $q^2-q+1$ & $SU(3,q)$  \\
 $q-1$ & $(q-1)^2$   & $Sp(4,q)$  \\
\hline
\end{tabular} 
\end{center}
\end{table}

\begin{table}[htdp]
\begin{center}
\caption{ Subgroups of ${^3}D_4(q)$ containing a Sylow $p$-subgroup\label{tab:Syl3D4}}
\begin{tabular}{ccc}
\hline 
 $p$ divides  & $p$-part of $G$ & Subgroup type containing a  \\
   & is $p$-part of & Sylow $p$ subgroup \\
 \hline
 $q^4-q^2+1$  & $q^4-q^2+1$ & see Lemma\\ 
 $q^2-q+1$ & $(q^2-q+1)^2$& $(q^2-q+1)\circ SU(3,q)$  \\
 $q^2+q+1$ & $(q^2+q+1)^2$ & $(q^2+q+1) \circ SL(3,q)$ \\
 $q+1$ & $(q+1)^2$ & $G_2(q)$\\
$q-1$ & $(q-1)^2$ & $G_2(q)$ \\
\hline
\end{tabular} 
\end{center}
\end{table}
\end{proof}
%
%
%

\begin{lemma} 
 If $G_0 \cong {^2}G_2^{\prime}(3^a)$, then $(x,G)$ cannot be a minimal counterexample. 
\end{lemma}
\begin{proof} 
We may assume that $a \ne  1$ since ${^2}G_2^{\prime}(3) \cong L(2,8)$. Since $q=3^a$ and by Lemma \ref{lem:outinvs}, we may assume that $x$ is semisimple. Now $|G_0|=q^3(q^3+1)(q-1)$ and the
maximal subgroups are given in \cite{G2}.  Since $p \nmid q$ there
are three mutually exclusive possibilities: $p \mid (q^2-1)$, $p
\mid q - \sqrt{3q}+1$, and $p \mid q + \sqrt{3q}+1$. First suppose
that $p \mid q^2-1$. Then a Sylow $p$-subgroup lies inside a
maximal subgroup $2 \times L(2,q)$, so $(x,G)$ cannot be a minimal counterexample in this case. 

 If $p \mid q^2-q+1$, then a Sylow $p$-subgroup is contained in
one of the abelian Hall subgroups of order $q \pm \sqrt{3q}+1$, so
we may assume that $x$ lies in one of these Hall subgroups and
that $|C_G(x)| = q \pm \sqrt{3q}+1$ (see part (4) of the main
theorem in \cite{Ward}). If $(x,G)$ is a minimal counterexample,
then $x$ can only be contained in a single subgroup $H$, which is either of type 
$\mathbb{Z}_{q + \sqrt{3q}+1}:\mathbb{Z}_{6}$ or $\mathbb{Z}_{q -
\sqrt{3q}+1}:\mathbb{Z}_{6}$. We choose an involution $y \in G$ that is not contained in $H$ and then $\langle x,y\rangle$ is not solvable.
\end{proof}

For the Suzuki groups, we will use a counting argument.
\begin{lem} \label{lem:countinv} Let $G$ be an almost simple group with socle $G_0$. Suppose that $X_1, \ldots X_k$ are representatives
  for the conjugacy classes of maximal subgroups of $G$ that contain $x$ and that do not contain $G_0$. Let
  $Y$ be the set of involutions in $G_0$. If
\begin{align} \label{eqn:countinv}
 |Y|> \sum_{i=1}^{k} \frac{|x^G \cap X_i||G:X_i||Y\cap X_i|}{|x^G|},
\end{align}
then there exists an involution $y$ in $G_0$ such that $\langle
  x,y \rangle$ contains $G_0$.
\end{lem}
\begin{proof}
  Suppose that $\langle x,y\rangle$ does not contain $G_0$ for all $y \in Y$. For each $i$, let $X_{i1}, \ldots X_{in_i}$ be the conjugates of
  $X_i$ that contain $x$. In particular, $n_i$ is the number of
  conjugates of $X_i$ that contain $x$. Thus we have
 \begin{align*}
      |Y \cap \bigcup_{i,j} X_{ij}| \le & \sum_{i=1}^k n_i|Y\cap X_i|.
  \end{align*}
  But
  \begin{displaymath}
    \frac{n_i}{|G:X_{ij}|}= \frac{|x^G \cap X_i|}{|x^G|}, 
  \end{displaymath}
  and thus
  \begin{displaymath}
   |Y \cap \bigcup_{i,j} X_{ij}| \le \sum_{i=1}^k \frac{|x^G \cap X_i||G:X_{i}||Y\cap X_i|}{|x^G|}. 
  \end{displaymath}
  This is a contradiction, since we assumed that $Y = Y \cap \bigcup_{i,j}X_{ij}$.
\end{proof}

\begin{lemma} 
 If $G_0 \cong {^2}B_2(2^a)$ then $(x,G)$ cannot be a minimal counterexample. 
\end{lemma}
\begin{proof} 
There are $(q^2+1)(q-1)$ involutions in $G_0={^2}B_2(q)$, and the maximal subgroups of $G_0$ are given in \cite{2B2}. If $x$
is inner-diagonal, then since $p$ is odd, there are three mutually exclusive 
possibilities: $p \mid q-1$, $p \mid q+\sqrt{2q}+1$, and $p \mid
q-\sqrt{2q}+1$. We will show that $(x,G)$ cannot be a minimal
counterexample using Lemma \ref{lem:countinv}. If $p \mid q-1$, then
the maximal subgroups that could contain $x$ are  Frobenius groups
of order $q^2(q-1)$,  dihedral groups of order $2(q-1)$ and
${^2}B_2(2)$. If $X_i$ is a Frobenius group, then $|Y \cap
X_i|=q-1$, by \cite[Theorem 2]{2B2}, and
\begin{displaymath}
  \frac{|x^G \cap X_i||G:X_i||Y\cap X_i|}{|x^G|} \le
  \frac{(q^3-q^2-q)(q^2+1)(q-1)}{q^2(q^2+1)} \le q^2-q-1.
\end{displaymath}
It follows that
\begin{displaymath}
  \sum_{i=1}^{k} \frac{|x^G \cap X_i||G:X_i||Y\cap X_i|}{|x^G|} \le
  (q^2-q-1)+\frac{(q-1)^2}{2} + (q-1),
\end{displaymath}
which is less than the number of involutions $(q^2+1)(q-1)$ in $G_0$ for $q \ge 8$. If $p \mid q+\sqrt{2q}+1$, then $x$ could be contained in a group
$\mathbb{Z}_{(q+\sqrt{2q}+1)}:[4]$ or ${^2}B_2(2)$, thus
\begin{displaymath}
  \sum_{i=1}^{k} \frac{|x^G \cap X_i||G:X_i||Y\cap X_i|}{|x^G|} \le
  (q+\sqrt{2q}+1)^2   + (q+\sqrt{2q}+1)
\end{displaymath}
and this is less than  $(q^2+1)(q-1)$ for $q \ge 8$. Similarly, if $p \mid q-\sqrt{2q}+1$, then $x$ could be contained
in a group $\mathbb{Z}_{(q-\sqrt{2q}+1)}:[4]$ or ${^2}B_2(2)$, so
\begin{displaymath}
  \sum_{i=1}^{k} \frac{|x^G \cap X_i||G:X_i||Y\cap X_i|}{|x^G|} \le
  (q-\sqrt{2q}+1)^2+(q-\sqrt{2q}+1).
\end{displaymath}
So $(x,G)$ is not a minimal counterexample when $x$ is
inner-diagonal.

If $x$ is a field automorphism, then we can use
the same counting argument as for the case $G_0=PSL(2,q)$. Indeed,
we would like to show that the right hand side of 
\begin{align} \label{eqn:2b2calc}
  |\Gamma| \le \sum_{i=1}^{k-1}
  \frac{i_2(X_i)|C_{G_0}(x)|}{|C_{X_i}(x)|}+|C_{G_0}(x)|
   \end{align}
is less than the number of involutions $(q^2+1)(q-1)$ in $G_0$. 
The possibilities for the maximal subgroups of $G_0$ containing
$\langle x,y \rangle \cap G_0$ are a Frobenius group of
order $q^2(q-1)$, a dihedral group of order $2(q-1)$,
the normalizer of a cyclic group
$\mathbb{Z}_{(q-\sqrt{2q}+1)}:[4]$, the normalizer of a
cyclic group $\mathbb{Z}_{(q+\sqrt{2q}+1)}:[4]$, and the
centralizer of $x$, ${^2}B_2(q^{1/p})$.  We label these subgroups $X_1$, $X_2$, $X_3$, $X_4$, and $X_5$ respectively.


Therefore,  if $q_0:=q^{1/p}$, then
\begin{align*}
   \sum_{i=1}^{k-1}
   \frac{i_2(X_i)|C_{G_0}(x)|}{|C_{X_i}(x)|}\!+\!|C_{G_0}(x)| \le&
   \frac{(q_0^p-1)q_0^2(q_0^2+1)(q_0-1)}{q_0^2(q_0-1)}\\
& +
   \frac{(q_0^p-1)q_0^2(q_0^2+1)(q_0-1)}{2(q_0-1)} \\
   & +
   \frac{3(q_0^p-\sqrt{2q_0^p}+1)q_0^2(q_0^2+1)(q_0-1)}{4(q_0+\sqrt{2q_0}+1)}\\
   & +\frac{3(q_0^p+\sqrt{2q_0^p}+1)q_0^2(q_0^2+1)(q_0-1)}{4(q_0-\sqrt{2q_0}+1)} \\
   & + q_0^2(q_0^2+1)(q_0^2-1) \\
   \le &  (q_0^p-1)(q_0^2+1) + \frac{(q_0^p-1)q_0^2(q_0^2+1)}{2} \\
   & + \!2(q_0^p\!+\!(2q_0^p)^{ \frac{1}{2}} \! +\!1)q_0^2(q_0\!+\!(2q_0)^{ \frac{1}{2}}\!+\!1)(q_0 \!- \!1) \\
   & + q_0^2(q_0^2+1)(q_0^2-1).
 \end{align*}
Since $p \ge 5$, an elementary calculation shows that (\ref{eqn:2b2calc}) holds;
thus $(x,G)$ cannot be a minimal counterexample.
\end{proof}

\begin{lemma} 
  Let $G_0$ be a sporadic group. Then $(x,G)$ cannot be a minimal counterexample. 
  \end{lemma}
 \begin{proof} 
  We can verify the sporadic groups in GAP using the character table library. We use Thompson's result  \cite[Corollary 3]{Thompson} that a group $H$ is nonsolvable if and only if there exists $a,b,c \in H$ of pairwise coprime order such that $abc=1$. Using the character table,  we can check that for any $x$ of prime order $p \ge 5$, there exists an involution $y$ such that $yx$ has order coprime to $2$ and $p$.  
\end{proof}
This completes the proof of Theorem \ref{inv}.

\section{Proof of Theorem 1.6}

Note that if there is a unique class of 
involutions, then $G$ cannot be a minimal
counterexample since Malle, Saxl and Weigel   \cite{MSW} prove that there exist three involutions in $G_0$ that 
generate $G_0$ unless $G_0 \cong PSU(3,3)$. Also, Guralnick and Saxl \cite{GS} prove that there exist three conjugates of any involution in an almost simple group that generate a subgroup containing the socle when the Lie rank is small. We will appeal to both of these result throughout the proof.


\begin{lemma} 
  Suppose that $G_0 \cong A_n$. Then $(x,G)$ cannot be a minimal counterexample. 
\end{lemma}

\begin{proof} 
Suppose that $(x,G)$ is a minimal counterexample with $G_0 \cong A_n$.
 We may assume by minimality that one of the following four cases hold: (i) $x=(12)(34)$ and $n=5$; (ii) $x=(16)(25)(34)$
and $n=7$;  (iii) $x=(12)(34)(56)(78)$ and $n=8$; (iv) $x$ is an automorphism of $A_6$ not contained in $S_6$.  In case (i),
let $g_1=(12345)$, $g_2=(345)$ so that $xx^{g_1} =(13542)$,
$xx^{g_2}= (354)$ and so $\langle x,x^{g_1},x^{g_2}\rangle \cong
A_5$. In case (ii), let $g_1=(1743526)$ and $g_2=(23654)$ so that
$xx^{g_1}=(1234567)$, $xx^{g_2}=(12)(56)$, and $\langle
x,x^{g_1}x^{g_2} \rangle \cong S_7$. In case (iii), let
$g_1=(143)(28567)$, $g_2=(13)(265874)$, then
$xx^{g1}=(1574)(2386)$, $xx^{g2}=(375)(468)$, and
$xx^{g2}x^{g1}=(1364725)$ and thus $\langle x,x^{g_1},x^{g_2}
\rangle \cong PSL(2,7)$ \cite{Atlas}. It is straightforward to rule out case (iv) in MAGMA.
\end{proof}

\begin{lemma} 
Suppose that $G_0$ is a simple group of Lie type and $G_0 \triangleleft G \le\Inndiag(G_0)$. Then $(x,G)$ is not a minimal counterexample.  
\end{lemma}

\begin{proof} 
  If the
twisted Lie rank of $G_0$ is 1, then $G_0 \cong PSL(2,q)$,
$PSU(3,q)$, ${^2}B_2(2^a)$, or ${^2}G_2(3^a)$. For all of these
groups except $PSL(2,q)$, there is a unique class of involutions in ${\rm Inndiag}(G_0)$ and so \cite{MSW} implies that $(x,G)$ is not a
minimal counterexample unless $G_0 \cong PSU(3,3)$.
And if $G_0 \cong PSL(2,q)$, then there exist three conjugates
that generate a group  containing $G_0$ by \cite{GS}. 
So we may
assume that the twisted Lie rank of $G_0$ is at least 2.

 First suppose that $q \ge 4$. If $x$ is contained in a maximal parabolic
subgroup, then we may assume that $x$ acts nontrivially on the
components of the Levi complement, which are not contained in
Table \ref{table:exceptions} except for those of type $PSL(2,4)$, $PSL(2,5)$,
and $PSL(2,9)$. In fact, the involution $x$ is always contained in a
parabolic  subgroup. Indeed, if $q$ is even, then $x$ is unipotent and if $q$ is
odd, then $x$ is semisimple and $x$ always centralizes a unipotent element $u$ 
(see the list of semisimple involutions and their centralizers \cite[Table
4.5.1]{GLS}) and the Borel--Tits Theorem implies that 
$C_{G}(u)$ is contained in a parabolic subgroup of $G$. 
  Now any
maximal parabolic subgroup has an almost simple component that is
not of type $PSL(2,4)$, $PSL(2,5)$, or $PSL(2,9)$ unless 
\begin{align*}
 G_0 
\in  \{ & A_2(q), {^2}A_3(q), 
 B_2(q), B_3(q), C_3(q), D^{\pm}_4(q), {^3}D_4(q),\\
&  G_2(q) \mid q =4,5,\!\!\mbox{ or } 9 \}.
\end{align*}
 However we can eliminate the groups that have a
unique classes of involutions in ${\rm Inndiag}(G_0)$. So if $q \ge 4$, $x\in {\rm Inndiag}(G_0)$ and $(x,G)$ is a minimal counterexample, then
$G_0 \cong {^3}D_4(4)$, ${^2}A_3(q)$, $B_2(q)$, $B_3(q)$, $C_3(q)$,
$G_2(4)$, or $D_4^{\pm}(q)$ for $q=4$, $5$, or $9$. But \cite[Lemma 3.2]{GS} and the proof of \cite[Lemma 3.4]{GS} eliminate $A_2(q)$ and ${^2}A_3(q)$. Now a
(unipotent) involution in ${^3}D_4(4)$ is contained in a subfield
subgroup ${^3}D_4(2)$ (see \cite{Spaltenstein}) 
 and so we can eliminate the case $G_0\cong{^3}D_4(4)$. If $G_0 \cong G_2(4)$, then any involution is contained in a subgroup of type $SL^{\pm}(3,q)\!:\!2$ (see for example \cite[Proposition 5.6]{GS}). If $G_0 \cong B_3(4)\cong
C_3(4)$, then we may assume that $x$ is contained an end-node
parabolic subgroup $P$ and not in the unipotent radical of $P$; the Levi complement will be of type $C_2(4)$ or $A_2(4)$. Thus
$G_0 \not\cong C_3(4)$. The same reasoning eliminates
$D^{\pm}_4(4)$. Thus the remaining possibilities for $q \ge 4$ are
\[ G_0
\in \{ B_2(4), B_2(5), B_2(9), B_3(5), B_3(9),
 C_3(5),C_3(9),  D^{\pm}_4(5),D^{\pm}_4(9), G_2(4)\}.
\]
 If $G_0 \cong C_3(5)$ or $C_3(9)$, then
\cite[Propostion 1.5]{LS} shows that $x$ stabilizes an orthogonal
decomposition $V = W \oplus W^{\perp}$, where $\dim W=4$ and $x$
acts noncentrally on $W$. So we can reduce to the case $C_2(5)$ or $C_2(9)$ and $(x,G)$ cannot be a
minimal counterexample. Similarly if $G_0 \cong B_3(5)$ or
$B_3(9)$, by \cite[Propostion 1.5]{LS}, $x$ stabilizes an
orthogonal decomposition $V = W \oplus W^{\perp}$, where $\dim
W=5$ or $6$ and $x$ acts noncentrally on $W$. 
 Since $\POm(5,q)$ and $\POm^{\pm}(6,q)$ for $q=5$ and $9$ are not listed in Table \ref{table:exceptions}, $(x,G)$ cannot be a minimal counterexample.
Similarly if $G_0 \cong D_4^{\pm}(5)$ or  $D^{\pm}_4(9)$, then $x$ will stabilize an orthogonal decomposition
as above where $\dim W =6$ or $7$ and $x$ acts noncentrally on $W$. 
Since
$\POm^{\pm}(6,q)$ and $\POm(7,q)$ are not listed in Table \ref{table:exceptions} for $q=5$ or $9$,  $(x,G)$ cannot be a minimal
counterexample. So for $q \ge 4$ it remains to consider
 \[
G_0 \in \{ B_2(4), B_2(5), B_2(9) \}. 
\]
 We can check in MAGMA that the theorem holds for these groups. 
 
Now suppose that $q=3$. Then $x$ is contained in a maximal
parabolic subgroup, and by Lemma \ref{lem:2.2}, we may assume
that it acts noncentrally on all of the Levi components. To prove
that $G$ is not a minimal counterexample, it suffices that one of
the Levi components is not in Table \ref{table:exceptions} and is
not solvable. Thus the only possibilities with $q=3$ are 
\begin{align*}
G_0 \in  \{ & A_2(3), A_3(3), {^2}A_2(3), {^2}A_3(3), {^2}A_4(3), {^2}A_5(3), B_n(3),\\
&   C_3(3), C_4(3), D_4^{\epsilon}(3),
{^3}D_4(3), G_2(3), {^2}G_2(3) \}
\end{align*}
 If we eliminate the cases
where there is a unique class of inner involutions, then we may assume that 
\begin{align*}
G_0 \in \{ & A_3(3), {^2}A_2(3), {^2}A_3(3), {^2}A_4(3), {^2}A_5(3),  \\ & B_n(3), C_3(3), C_4(3), D_4^{\epsilon}(3)\}.
\end{align*}
If $G_0 \cong B_n(3)$, then by
\cite[Proposition 1.5]{LS}, 
 $x$ stabilises an orthogonal decomposition
$W \oplus W^{\perp}$, where $W$ is chosen so that $\dim
W^{\perp}$ is minimal, and is therefore at most 2. If $n \ge 4$ and $x$ is not a reflection
then we can choose $W$ so that  $x$ acts on $W$  noncentrally and not as a reflection.  Thus $(x,G)$ cannot be a minimal
counterexample if $G_0=B_n(3)$ and $n \ge 4$. Checking the remaining cases in MAGMA shows that there are no minimal
counterexamples when $q=3$.

Now suppose that $q=2$ so that $x$ is unipotent. Then we can use
Lemma \ref{lem:2.2}. If $G_0=A_n(2)$, then we may assume that
$x$ is contained in a maximal end-node parabolic subgroup $P$ and
not contained in the unipotent radical of $P$. Thus we can reduce
to the case that $x \in A_{n-1}(2)$ since there are no {\it
inner} exceptions in $A_{n-1}(2)$. This argument thus reduces to the
case of $G=PSL(3,2) \cong PSL(2,7)$, and \cite{GS} shows that
$(x,G)$ is not a counterexample. Now suppose that $G_0 \cong
{^2}A_n(2)$. Then \cite[Proposition 1.4]{LS} shows that $x$ stabilizes an
orthogonal decomposition $V=W \oplus W^{\perp}$ such that $W$ has codimension $1$ or $2$, 
 and if $x$ is not a unitary transvection, we can choose $W$ such that $x$ does not act on  $W$ trivially, or as a transvection.
Thus if $n \ge 6$, then
there are no minimal counterexamples with $G_0 \cong {^2}A_n(2)$. We can check in
MAGMA that the only exceptions when $n \le 5$ are the unitary transvections.

 Now suppose that $G_0 \cong B_n(2) \cong C_n(2)$ or $D_n^{\pm}(2)$. Suppose that
$x$ is not a transvection. If $n \ge 5$, then we can take an
orthogonal decomposition $V=W \oplus W^{\perp}$ as above using \cite[Proposition
1.4]{LS} such that $x$ does not act on $W$ trivially or as a transvection, and $\dim W \ge 6$.
Thus $(x,G)$  cannot be a minimal
counterexample when $n \ge 5$. Note that $Sp(4,2) \cong S_6$ so
we can just verify in MAGMA that $Sp(6,2)$, $Sp(8,2)$, and  $\Omega^{\pm}(8,2)$ cannot be
counterexamples to the theorem.

To complete the analysis when $q=2$, we eliminate the cases
$G_0 \cong {^3}D_4(2)$, $E_6(2)$, ${^2}E_6(2)$, $E_7(2)$,
$E_8(2)$, $F_4(2)$, ${^2}F_4(2)$,  and $G_2(2)$ using MAGMA.  
\end{proof}

\begin{lemma} 
Suppose that $G_0$ is a simple group of Lie type and $x$ is a field automorphism of $G_0$. Then $(x,G)$ is not a minimal counterexample.  
\end{lemma}
\begin{proof} 
Now suppose that $x$ is a field automorphism of order $2$.
By \cite[7.2]{GL}, we may assume that $x$ is a standard field automorphism.

First observe that if $G$ is a Suzuki--Ree group then all field automorphisms have odd order. 
So we may assume that $G$ is not a Suzuki--Ree group. Now
$q \ge 4$ and $x$ will act as a field
automorphism on a $SL(2,q)$ subgroup and so $(x,G)$ cannot be a
minimal counterexample unless $G_0=PSL(2,q)$ or $q=4$ or $9$. If $q=4$ or $9$,
then we may assume that $x$ acts as a field automorphism on a subgroup of type
$A_2(q)$, $B_2(q)$, or $G_2(q)$ or ${^3}D_4(q)$. We can eliminate the first two cases in MAGMA. If $G_0 \cong G_2(q)$ or ${^3}D_4(q)$, then $x$ normalizes but does not centralize a subgroup of type $SL(3,q)$ or $SL(2,q^3)$ respectively. So it remains to treat the
 case where $x$ is a field automorphism of $PSL(2,q)$. If $q$ is
 even, then since $q\ne 4$, we have $q=q_0^2$ where $q_0 \ge 4$.
 But, then there exist $y$ of order $q_0-1$ and $z$ of order
 $q_0+1$ such that $x$, $xy$, and $xz$ are conjugate,
  and $y$ and
 $z$ do not commute and thus $\langle x,xy,xz \rangle$ contains
 $PSL(2,q_0)$, which is not solvable. If
 $q$ is odd, then suppose that $\mathbb{F}_q^* = \langle w
 \rangle$. Let $\lambda = w^{\frac{(q_0+1)}{2}}$ so that
 $\lambda^{q_0}=-\lambda$. Then $x$ inverts
\[ \left (
 \begin{matrix}
 1& \lambda \\
 0 & 1
 \end{matrix} \right) \]
and
\[ \left( \begin{matrix}
 1& 0 \\
 \lambda & 1
 \end{matrix}  \right) \]
 thus there exist conjugates of $x$, $y$ and $z$ say such that
 $xy$ and $xz$ are the transvections above. It follows that $\langle
 x,y,z\rangle$ contains a subgroup of type $PSL(2,q_0)$ which is
 not solvable since $q > 9$. \\
  \end{proof}

\begin{lemma} 
Suppose that $G_0$ is a simple group of Lie type and $x$ is a graph-field automorphism of $G_0$. Then $(x,G)$ is not a minimal counterexample.  
\end{lemma}
\begin{proof} 
 Now suppose that $x$ is a graph-field automorphism so that $G_0$ is an untwisted simple group of Lie type. By
 \cite[7.2]{GL}, we may assume that $x$ is a standard graph-field
 involution. Define $q_0$ by $q=q_0^2$ as before. 
 
  If $G_0 \cong PSL(d,q)$, with $d\ge 3$, then $x$ normalizes a subfield subgroup $PSL(d,q_0)$ (acting as a graph automorphism), which is not an exception unless
 $G_0 \cong PSL(4,4)$, since $PSL(4,2)$ is in Table \ref{table:exceptions}. We can eliminate this case in MAGMA. 
  If $G_0 \cong D_m(q)$ and
  $m \ge 4$, then $x$ will act as a graph field automorphism
 on a $D_{m-1}(q)$ subgroup. 
 Similarly a graph-field automorphism of $E_6(q)$ will act as a
 graph-field automorphism on $A_5(q)$. 
  If $G_0 \cong F_4(2^a)$, $G_2(3^a)$, or $B_2(2^a)$ then the extraordinary `graph' automorphism squares to the generating field automorphism.  So there are involutary graph-field automorphisms (in the sense of \cite[2.5.13]{GLS})  only when $a=1$;  these cases are easily eliminated in MAGMA. 
 \end{proof}

\begin{lemma} 
Suppose that $G_0$ is a simple group of Lie type and $x$ is a graph automorphism of $G_0$. Then $(x,G)$ is not a minimal counterexample.  
\end{lemma}
\begin{proof} 
We use the terminology of \cite[2.5.13]{GLS}. So there are no graph automorphisms of $F_4(2^a)$, $G_2(3^a)$, or $B_2(2^a)$, 
and no graph-field or graph automorphisms of the Suzuki--Ree groups. 
 
If $G_0 \cong L^{\epsilon}(3,q)$, then \cite[Lemmas 3.2 and 3.3]{GS} show that $(x,G)$ cannot be a minimal counterexample.

If $G_0 \cong L^{\epsilon}(4,q)$, then observe that $G_0 \cong \POm^{\epsilon}(6,q)$, and an involutory graph automorphism of $L^{\epsilon}(4,q)$ is contained in $PCO^{\epsilon}(6,q)$. By \cite[Propositions 1.4 and 1.5]{LS} for example, $x$ will normalize (and not centralize) a subgroup of type $\POm(5,q)$, $\POm^{\pm}(4,q)$, or $\POm(3,q)$. Therefore $(x,G)$ cannot be a minimal counterexample unless $q=2$ or $q=3$. We can verify that Theorem \ref{thm:3invs} holds for $q=2$ and $3$ using MAGMA.  


Suppose that $G_0\cong L^ \epsilon(d,q)$, $q$ is odd, and $d \ge 5$.
The class representatives of graph involutions in this case are
given in \cite[Table 4.5.1]{GLS} and \cite[3.9]{LS}. We can deduce from
these representatives that if $d$ is even, then $x$ will act as a
graph automorphism on a type $PSL(d-1,q)$ or $PSL(d-2,q)$
subgroup; if $d$ is odd, then $x$ will act as a graph
automorphism on a type $PSL(d-1,q)$ subgroup. So since $(x,G)$ is
a minimal counterexample, we can reduce to the case $G_0 \cong L^ \epsilon(4,q)$ since $L^{\epsilon}(3,q)$ has been eliminated. 

Now suppose that $G_0\cong L^ \epsilon(d,q)$, $q$ is even, and $d \ge 5$.
The class representatives for graph involutions when $q$ is even can be found in \cite[Lemma 3.7]{Moufang}. There are two classes when $d$ is even and one class when $d$ is odd. In all cases, $x$ normalizes a subgroup of type $L^{\epsilon}(d-1,q)$ or $L^{\epsilon}(d-2,q)$, acting as a graph automorphism. So since $(x,G)$ is
a minimal counterexample, we may assume that $G_0 \cong L^ \epsilon(3,q)$,$L^ \epsilon(4,q)$, $L^ \epsilon(5,2)$, or $L^ \epsilon(6,2)$. We can eliminate the last two possibilities using MAGMA.

If $x$ is an involutory graph automorphism of $P\Omega^{\pm}(d,q)$, then $x \in PCO^{\pm}(d,q)$ and we may assume that $d \ge 8$. By \cite[Lemma 3.3]{LS} for example, we may assume that $x$ normalizes but does not centralize a subgroup of type $\POm^{\epsilon}(d-b,q)$, where $b \le 4$. 
Thus $(x,G)$ cannot be a minimal counterexample unless $q=2$, or $q=3$. But if $d \ge 10$ and $x$ is not an orthogonal transvection or reflection, then we may assume that $x$ does not act as orthogonal transvection or reflection in the type  $\POm^{\epsilon}(d-b,q)$ subgroup. We can eliminate the groups $G_0 \cong \POm^{\epsilon}(8,3)$ and $\POm^{\epsilon}(8,2)$ in MAGMA.  

If $G_0 \cong E^{\epsilon}_6(q)$ and $x$ is an involutary graph automorphism, then $x$ normalizes but does not centralize a subgroup of type $F_4(q)$; thus $(x,G)$ cannot be a minimal counterexample. This follows from an analysis of the standard class representatives of graph involutions found in \cite[Lemma 3.6]{Moufang} for $q$ odd and \cite[19.9]{AS} for $q$ even. See the proof of \cite[Proposition 5.2]{GS} for example.
\end{proof}

\begin{lemma} 
 If $G_0$ is a sporadic group then $(x,G)$ is not a minimal counterexample.  
\end{lemma}
\begin{proof} 
If $G_0$ is not the Monster group $M$ or the Baby Monster group $B$, then we can use MAGMA. If there is a unique class of involutions then $(x,G)$ cannot be a minimal counterexample by \cite{MSW}. In the other cases, we use the representations in \cite{wwwatlas} together with the representatives for the conjugacy classes of involutions. We search at random for conjugates $y$ and $z$ of $x$ and test the subgroup $H= \langle x,y,z\rangle$ for nonsolvability (by searching for $a,b,c \in H$ of pairwise coprime order such that $abc=1$ \cite{Thompson}). 

If $G=B$, then there are 4 classes of involutions.  The centralizer order of an element in class 2A in the Harada Norton group $HN$ is divisible by $5^3$, thus such an involution is contained in $B$-class 2B. Any involution in $Th$ is in $B$-class 2D by \cite[pg 4]{bmonster1}. We can verify in MAGMA that an element in 2A belongs to a triple of conjugates generating a nonsolvable group. If $x$ belongs to class 2C, then the character table of $G$ implies that there exist conjugates $y$ and $z$ such that $xy$ has order 19 and $xz$ has order 33. 
 By analyzing the maximal subgroups, $\langle x,y,z \rangle=B$.        \\
         There are two classes of involutions in the Monster group $M$.  If $x$ is in class 2A, then $x$ is contained in a subgroup isomorphic to $PSL(3,2)$; for an involution in $PSL(3,2)$ in the maximal subgroup $(PSL(3,2)\times Sp(4,4):2).2$ must be in class $2A$ since it centralizes an element of order 17 in $Sp(4,4)$ and 2B elements do not centralize elements of order 17. Any involution in $PSU(3,8)$ is in $M$-class 2B by \cite[4.5]{anatomy}; thus neither class of involutions can be involved in a minimal counterexample.    
\end{proof}
This completes the proof of Theorem \ref{thm:3invs}.

\section{Proof of Theorem 1.4}
\subsection{Order 9} First suppose that $x$ has order $9$.
Suppose that $G_{0}\!\cong \!\PSL(d,3)$, $\PSU(d,3)$, or $\PSp(d,3)$ and $x^{3}$ is a transvection. Then $x$ lifts to an element in $\SL(d,3), \SU(d,3),$ or $\Sp(d,3)$, with Jordan form $J_{4}J_{3}^{r_{3}}J_{2}^{r_{2}}J_{1}^{r_{1}}$ (and $r_{3}$ and $r_{1}$ are even in the symplectic case). It is well known that in the linear and unitary cases, $x$ must be contained in a subgroup of the form $\SL(4,3) \times \SL(d-4,3)$ or $\SU_{4}(4,3) \times \SU(d-4,3)$, where $x$ has order $9$ in the first factor. In the symplectic case, we can also assume that $x \in \Sp(4,3) \times \Sp(d-4,3)$ and $x$ has order $9$ in the first factor, for example by \cite[Thm 2.12]{LSgood}. 

Similarly, if $G_{0}=\POm^{\epsilon}(d,3)$ and $x^{3}$ is a long root element, then $x$ lifts to an element in $\Om^{\epsilon}_{n}(3)$, which has Jordan form $J_{4}^{2}J_{3}^{r_{3}}J_{2}^{r_{2}}J_{1}^{r_{1}}$ or $J_{5}J_{3}^{r_{3}}J_{2}^{r_{2}}J_{1}^{r_{1}}$. Again by \cite[Thm 2.12]{LSgood} for example, we may assume that $x$ is contained in a subgroup of type
$\O^{\epsilon}(8,3) \times \O^{\epsilon}(d-8,3)$ or $\O(5,3) \times \O^{\epsilon}(d-5,3)$, in which $x$ has order $9$ in the first factor. Thus we have reduced to the cases $G_{0}= \PSL(4,3),\PSU(4,3)$, $\PSp(4,3) \cong \POm(5,3)$, or $\POm^{\epsilon}(8,3)$. It is easily verified in MAGMA that these groups, are not counterexamples, and that neither are the groups with $G_0 = G_2(3)$, ${^3}D_4(3)$, $D_4(3)\!:\!3$, or ${^2}G_2(3)$.

If $G_0$ is one of the other exceptional groups defined over $\mathbb{F}_3$, then we can find representatives for the conjugacy classes of order $9$ using \cite{Mizuno1, Mizuno2, Shoji} together with \cite[Tables D and 9]{Lawther}. Information on the unipotent conjugacy classes of ${^2}E_6(3)$ was provided by  Frank L\"{u}beck, which the author is most grateful for. In all cases, it is easily seen that $x$ is contained in an almost simple subgroup and thus cannot be a minimal counterexample.

Suppose $x^3 \in G$ is a pseudoreflection and $G_0=\PSU(d,2)$. First note that if $x \in \PGU(d,2)$ and $x$ does not lift to an element of order $9$ in $\GU(d,2)$, then the minimal polynomial of $x$ must be of the form $t^9+\mu$, where $\mu \in \mathbb{F}_{4}$, and $\mu \ne 0,1$. 
It follows that the eigenvalues of $x$ must all be $9$th roots of $\mu$, and the eigenvalues of $x^{3}$ must all be cube roots of $\mu$; in particular, $x^{3}$ cannot be a pseudoreflection.  Thus we may assume that $x$ lifts to an element of order $9$ in $\GU(d,2)$, with minimal polynomial $t^{9}-1$.  The eigenvalues of $x$ must all be $9$th roots of $1$, and since $x \in \GU(d,2)$, the eigenvalues are permuted by the map $\lambda \mapsto \lambda^{-2}$. Thus the \emph{primitive} $9$th roots of 1 must occur as eigenvalues of $x$ in triples. Suppose that the eigenvalues of the pseudoreflection $x^{3}$ are $a$ with multiplicity $1$ and $b$ with multiplicity $d-1$. It follows that $a=1$, $b$ is a primitive cube root of unity, $d \equiv 1 \imod 3$ and the eigenvalues of $x$ are $\phi_{1}$ with multiplicity 1 where $\phi_{1}^{3}=1$, and $\phi_{2}$, $\phi_{2}^{-2}$, $\phi_{2}^{4}$, each with multiplicity $\frac{d-1}{3}$, where $\phi_{2}$ is a primitive $9$th root of unity. Since $ \GU_{4}(2) \times \GU_{3}(2) \times \cdots \times \GU_{3}(2)$ contains an element with the same eigenvalues (and the same Jordan form) 
we may assume that $x$ is contained in this subgroup, and we have reduced to the case $G_{0}=\PSU(4,2)$. This case is easily eliminated using MAGMA.

\subsection{Order 6} 

By Theorem \ref{thmA*}, if $(x,G)$ is a minimal counterexample, then
$G_0$ is a simple group of Lie type and $q=3$ or $G_0 \cong
PSU(d,2)$. Moreover, $x^2$ is a long root element or a pseudoreflection
in each case respectively. We will consider the possible conjugacy
classes for $x$. \par

If $x \in {\rm Inndiag}(G_0)$, then let $x_s:=x^3$ and $x_u:=x^4$.
So $x=x_sx_u=x_ux_s$, and $x_s$ is semisimple and $x_u$ is
unipotent.
\begin{lem} \label{lift6}
If $G_0 \cong PSL(d,3)$, $PSp(d,3)$, or $PSU(d,3)$, $x$ is inner-diagonal and $(x,G)$ is a minimal counterexample, then $x$ lifts to an element of order
$6$ in $GL(d,3)$, $GSp(d,3)$, or $GU(d,3)$ respectively.
\end{lem}
\begin{proof}
In the linear and symplectic cases, the only  central elements are
$\pm I_d$, so either $(xz)^6=I_d$ for all central elements $z$, or
$(xz)^6=-I_d$ for all $z$. In the latter case, let $y=xz$. Then the
minimal polynomial of $y$ divides $t^6+1= (t^2+1)^3$. Moreover,
since $y^2$ is a scalar multiple of a transvection and $t^2+1$ is
irreducible over $\mathbb{F}_3$, it follows that the minimal
polynomial is $(t^2+1)^2$. However, this would imply that $y^{2}$ had $(t+1)^2$ occuring twice as an invariant factor when it should  occur at most  once.  So $x$
lifts to an element of order $6$ in the linear and symplectic
cases. \par

In the unitary case, let $i \in \mathbb{F}_{3^{2}}$ be a primitive $4$th root of unity so that $Z( \GU(d,3)) = \langle i I_{d} \rangle$.  If $(xz)^6=-I_d$, then $(xzi)^6=-I_d(-1)^3=1$ and $x$ lifts to an order $6$ element. The
only other possibility is that $(xz)^6= \pm i I_d$, in which case
the minimal polynomial of $y=xz$ would divide $(t^2 \pm i)^3$. As
before, since $y^2$ is a scalar multiple of a transvection, the
minimal polynomial of $y$ would divide $(t^2 \pm i)^2$. Now $t^2 \pm i$ is irreducible, and if $y$ had minimal polynomial $t^2 \pm i$, then $y$ would have projective order $2$. Thus the only possibility is
that $m_y(t)=(t^2 \pm i)^2$, but then $(t \pm i)^2$  would occur twice as an invariant factor of
$y^2$.  So $x$ lifts to an element
of order $6$ in the unitary case as well.
\end{proof}
\begin{lem}
If $G \le PGL(d,3), PGSp(d,3)$ or $PGU(d,3)$ and $x \in G$ has order $6$, then there exists $g
\in G$ such that $  \langle x,x^g \rangle$ is not solvable.
\end{lem}
\begin{proof}
By Lemma \ref{lift6}, we can lift $x$ to an element of order $6$
in $GL(d,3)$, $GSp(d,3)$, or $GU(d,3)$.  Since $x=x_ux_s=x_sx_u$ and $x_u$ is a
transvection, the minimal polynomial of $x$ divides $(t^2-1)^2$.
Now $x^2$ is a transvection and its invariant factors are $(t-1)^2$ with multiplicity one and $(t-1)$ with multiplicity $d-2$; thus the minimal polynomial of $x$ is not $(t^2-1)^2$ otherwise the multiplicity of $(t-1)^{2}$ in the invariant factors of $x^{2}$ would be at least 2. In fact, we can show that the invariant factors of $x$ must be 
 $t+ \epsilon_{1}$  with multiplicity $m_{1}$, $t^{2}-1$  with multiplicity $m_{2}$, and $(t^{2}-1)(t- \epsilon_{2})$ with multiplicity $1$, where $ \epsilon_{i} \in \{ \pm 1\}$. In the linear and unitary cases, there exists $y \in \GL^{\pm}(3,3) \times \GL^{\pm}(d-3,3)$ with $y$ of order $6$ in the first factor, having  the same invariant factors as $x$. Thus $x$ and $y$ are conjugate (see \cite{Wall} for example) and we can reduce to the cases $G_0 = \PSL(3,3)$  and $\PSU(3,3)$. 
 
 In the symplectic case we consider the elementary divisors of $x$, which must be $(t- \epsilon)^2$ with multiplicity $1$, $(t+ \epsilon)$ with multiplicity $m_1 \ge 1$, and $(t- \epsilon)$ with multiplicity $m_2$ ($\epsilon= \pm 1$).
By considering the vector space $V$ as an $\mathbb{F}_{3}\langle x \rangle$-module, we can see that $V$ decomposes as 
\[V = U \oplus U' \cong \mathbb{F}_{q}(t)/ (t - \epsilon)^2 \oplus \left(\mathbb{F}_{q}(t)/ (t+\epsilon) \oplus \cdots \oplus \mathbb{F}_{q}(t)/ (t\pm 1)\right).\]
Since $x^2$ is a symplectic transvection, there exists $t \in U$, and $\lambda \in \mathbb{F}_{3}$ such that $x^2: v \to v+ \lambda(t,v)v$ for all $v \in V$. Moreover, since $U \cong \mathbb{F}_{q}(t)/ (t - \epsilon)^2$, there exists $u \in U$ such that $(u,t)\ne 0$ and thus $U = \langle u,t \rangle$ is a $2$-dimensional, nondegenerate subspace. Now consider $V = U \perp U^{\perp}$. Observe that $x$ has order $1$ or $2$ on $U^{\perp}$ 
and in particular $x$ has a semisimple action on $U^{\perp}$. Thus $U^{\perp} = \oplus U_i$ where the $U_i$ are nondegenerate $x$-invariant $1$ or $2$ dimensional subspaces. Moreover, there exists an $x$-invariant decomposition $V = (U \oplus U_j) \oplus (U \oplus U_j)^{\perp}$  into nondegenerate subspaces of dimension $4$ and $n-4$, and $x$ has order $6$ on $(U \oplus U_j)$; thus we can reduce to the case $G_0=\PSp(4,3)$. We can easily verify in MAGMA that $G_0= PSL(3,3)$, $\PSU(3,3)$ and $\PSp(4,3)$ are not counterexamples to the theorem. 
\end{proof}

\begin{lem}
    If $G$ is an orthogonal group and $x$ has projective order $6$, then there exists $g \in G$ such that $\langle x,x^g  \rangle$
    is not solvable.
\end{lem}
\begin{proof}
If $x$ is contained in $PCO^{\epsilon}(d,3)$, where $x^2$ is a
long root element and $x^3$ is an involution, then $x$ will lift to an element $y$ of order $6$ or $12$ in $CO^{\epsilon}(d,3)$. \par

Now $l = y^{ \frac{|y|}{3}}$ is a long root element, and all long root elements are conjugate, so we may assume that 
\begin{displaymath} 
l : v \to v + (v,e_2)e_1 - (v,e_1)e_2,
\end{displaymath} 
where $\{e_1, f_1, e_2,f_2, \ldots \}$ is a basis for $V$ and $(e_i,f_j) = \delta_{ij}$ and $(e_i,e_j) = (f_i,f_j)=0$.

Since the elementary divisors for $l$ are $(t-1)^2$ with multipiclity $2$ and $(t-1)$ with multiplicity $n-2$, the possibilities for the elementary divisors of $y$ are as follows:
\begin{enumerate}
\item $(t - \epsilon)^2$ with multiplicity $2$, $(t+ \epsilon)$ with multiplicity $m_1$ ($m_1\ge 1$ if $\epsilon = 1$), and $(t- \epsilon)$ with multiplicity $m_2$ ($\epsilon= \pm 1$).
\item $(t-1)^2$ with multiplicity $1$, $(t+1)^2$ with multiplicity $1$, $(t+1)$ with  multiplicity $m_1$ and $(t-1)$ with multiplicity $m_2$. 
\item $(t^2+1)^2$ with multiplicity $1$, and $(t^2+1)$ with multiplicity $ \frac{n-4}{2}$.
\end{enumerate}
So as an $ \mathbb{F}_{3} \langle y \rangle$-module, 
\begin{displaymath} 
 V= U_1 \oplus U_2 \oplus \cdots \oplus U_k 
\end{displaymath} 
where $U_i \cong \mathbb{F}_3(t)/ f_i(t)$, and $f_i(t)$ is the $i$th elementary divisor of $y$. Set $U= U_1 \oplus U_2$ in the first and second cases and set $U=U_1$ in the third case; so $U$ is $4$-dimensional as a vector space. Now considering $U$ as an $ \mathbb{F}_{3}(l)$-module, we have $U =W_1 \oplus W_2$  where $W_i \cong \mathbb{F}_3(t)/ (t-1)^2$, and we claim that $U$ is a nondegenerate subspace of $V$. For there exists $v_1 \in W_1$ such that $v_1^l - v_1 = \lambda_1 e_1 + \lambda_2 e_2 \ne 0$ and similarly  there exists $v_2 \in W_2$ such that $v_2^l - v_2 = \mu_1 e_1 + \mu_2 e_2 \ne 0$.
Since $\lambda_1 e_1 + \lambda_2 e_2$ and $\mu_1 e_1 + \mu_2 e_2$ are linearly independent, it follows that there exist constants $a_1,a_2,b_1,b_2$ such that $v_1' = a_1 v_1 + a_2 v_2$, $v_2' = b_1 v_1 + b_2 v_2$, and  $(v_i', e_j) = \delta_{ij}$ for $i, j \in \{ 1,2\}$. Now it is easy to check that $U = \langle e_1,e_2,v_1',v_2' \rangle$ is a nondegenerate space. For if 
\begin{displaymath} 
w=  a e_1 + b e_2 + c v_1' + dv_2' 
\end{displaymath}
is a degenerate vector in $U$, then $(w,e_1)=(w,e_2)=0$; so $c=d=0$. 
Thus $w = ae_1 + be_2$ and $(w,v_1')= (w,v_2')=0$; so $a=b=0$, $w=0$, and $U$ is nondegenerate.
Now  $V= U \oplus U^{\perp}$, and $y$ has projective order $1$ or $2$ on $U^{\perp}$. In particular $y$ has a semisimple action on $U^{\perp}$ and there exists a nondegenerate $y$-invariant subspace $U' \le U^{\perp}$ of dimension $1$ or $2$ such that $y$ has projective order $6$ on $U \perp U'$. Thus we may assume that $n \le 6$, but then $G$ is isomorphic to a linear, unitary or symplectic group.
\end{proof}

\begin{lemma}
Suppose that $G \le PGU(d,2)$  $(d \ge 4)$, that $x$ has order $6$,
and that $x^2$ is a pseudoreflection.  Then there exists $g \in G$ such that $\langle x,x^g\rangle$ is nonsolvable.
\end{lemma}
\begin{proof} 
 First observe that $x$ lifts to an element $y$ of order $6$ in $\GU(d,2)$. Indeed, for any lift $y$ of $x$, we have $y^2 = z r$ where $z \in Z(GU(d,2))$, and $r$ is a pseudoreflection. But then $y^6 = (zr)^3 = 1$ since $Z(GU(d,2))$ has order $3$.


Since $x^2$ is a pseudoreflection,  $C_{\GU(d,2)}(x^2) \cong \!\GU(1,2)\! \times \!\GU(d-1,2)$ by \cite[Lemma 4.1.1]{KL}, and $x \! \in \! C_G(x^2)$.  Moreover, by multiplying $x$ by $z \in Z(\GU(d,2))$ if necessary, we may assume that $x$ has order $3$ on the first component, and $2$ on the second component. Now in $GU(d-1,2)$, two unipotent elements are conjugate if and only if they have the same Jordan form(see \cite{Wall} or \cite[Theorems 2.12 and 3]{LSgood}). But the Jordan form of an involution is of the form $J_2^r, J_1^{d-1-2r}$, and so there exists a conjugate of $x$ that is contained in $GU(1,2) \times GU(4,2) \times \GU(d-5,2)$, and on which $x$ has order $3$ on the first component, and $2$ on the second component. So we may assume that $4 \le d \le 5$, and we can eliminate these cases in MAGMA.
\end{proof}

\begin{lemma} 
 Suppose that $x$ is an inner-diagonal automorphism of an exceptional group defined over $\mathbb{F}_{3}$. Then $(x,G)$ is not a minimal counterexample.
\end{lemma}
\begin{proof} 
  We can verify the Lemma using MAGMA. For the
smaller groups, we can calculate the conjugacy classes directly. For the larger groups we can use the Groups of Lie type package in MAGMA to construct a Sylow $2$-subgroup $S$ of $C=C_G(y)$, where $y$ is a long root element. We can then calculate the conjugacy classes of $S$, to find the class representatives $s_1, s_2, \ldots s_k$ of involutions in $S$. Then every element $x$ of order $6$ in $G$ such that $x^2$ is a long root element is conjugate to at least one element in  $\{y s_1, ys_2, \ldots, ys_k\}$. Now for each $x = ys_i$, we can search for a random conjugate $x^g$ such that $\langle x,x^g\rangle$ is not solvable. 
\end{proof}

 \begin{lemma} 
 If $x \not\in \Inndiag(G_0)$, then $(x,G)$ is not a minimal counterexample. \end{lemma}
\begin{proof} 
 Suppose that $x^2$ is a transvection and $x^3$ is a graph
automorphism. If $G_0 \cong PSL(d,3)$ and $d$ is even, then there
are three classes of graph automorphism. Representatives for the
three classes are $\iota S$, $\iota S^{+}$ and $\iota S^-$ where
 $\iota$ is the inverse transpose automorphism. Their
centralizers are of type $Sp(d,3)$, $O^+(d,3)$ and $O^-(d,3)$
respectively. Now $x=x^4x^3$, and $x^4$ is a transvection so there exists $a \in
V$ such that
\begin{displaymath}
x^4 : v \rightarrow v+ f(v)a,
\end{displaymath}
where $f$ is a linear functional on $V$, $\dim \ker{f}=
n-1$, and $f(a)=0$. Now all of the graph automorphisms above
stabilize a subgroup of type $GL(d-2,3) \times GL(2,3)$.
Now we can  conjugate $x$ by
$h \in C_G(x^3)$ (of type $Sp(d,3)$, $O^{\pm}(d,3)$) so that $h(a) \in U$ where $U$ is the subspace
of $V$ corresponding to the subgroup $GL(d-2,3)$. 
 Thus we can consider $x$
acting as an automorphism on $GL(d-2,3)$. There is only one class of graph involutions
when $d$ is odd, with representative $\iota$, and centralizer of type $O(d,3)$. We can make the same
reduction here. So it suffices to deal with cases where $d \le 4$. It is
easy to eliminate these cases in MAGMA.\par
	
The case where $G_0\cong PSU(d,3)$ is very similar. In this case, $x^4$ is a unitary transvection 
and $x^3$ is a graph automorphism. The classes of involutory graph automorphisms are described in \cite[Table 4.5.1]{GLS}. When $d$ is even, there are three classes as in the linear case with centralizers of type $O^{+}(d,3)$, $O^{-}(d,3)$, and $Sp(d,3)$. These classes are described explicitly in \cite[pg 43]{Moufang} and \cite[pg 288]{LS}, and we see that each class normalizes a subgroup of type $GU(d-2,3) \times GU(2,3)$ or $GU(d-1,3)\times GU(1,3)$. In particular, there exists an $x^3$ invariant, nondegenerate subspace $U$ of $V$  of dimension $d-2$ or $d-1$. Moreover, as in the linear case we can take $h \in C_G(x^3)$ such that $h(a) \in U$, and so we may assume that $x^4$ acts a unitary transvection on $U$. Similarly, when $d$ is odd, there is only one class of graph involutions and \cite{Moufang} and \cite{LS} show that $x^3$ normalizes a subgroup of type $GU(d-1,3)$; thus we may assume that $x$ normalizes this subgroup and has order $6$ on it. Therefore, it suffices to check the cases $G_0=PSU(4,3)$ and $G_0=PSU(3,3)$ in MAGMA.  

Next, suppose that $G_0 \cong PSU(d,2)$, $x^3$ is an involutory graph automorphism, and $x^4$ is a pseudoreflection. If $d$ is odd, then there is one class of graph involutions, and we may therefore assume that $x^3$ acts as a standard field automorphism on matrix entries, with centralizer of type $O^{+}(d,2)$ (see \cite[19.9]{AS}). Thus, conjugating by $h \in C_G(x^3)$ if necessary, we may assume that $x$ will normalize a subgroup of type $GU(d-1,2)$, and $x$ will act as an element of order $6$ on it. 
 If $n$ is even, then there are two classes of graph involutions, with centralizers of type $Sp(d,2)$ and $C_{Sp(d,2)}(t)$ where $t$ is a transvection in $Sp(d,2)$. In the first case, we may assume that $x$ acts as a standard field automorphism on the matrix entries and so $x$ will normalize a subgroup of type $GU(d-1,2)$ as before. In the other case, the pseudoreflection $x^4$ is contained in $C_{G_0}(x^3)=C_{Sp(d,2)}(t)$; moreover, $C_{G_0}(x^3)$ is contained in a subgroup $Sp(d,2)\langle x^3 \rangle = Sp(n,2) \times \langle x^3t \rangle$ (see \cite[pg 288]{LS}),
  and $x^3$ acts nontrivially, as an inner automorphism on $Sp(d,2)$. Thus $x \in Sp(d,2)\langle x^3 \rangle$, which is a smaller almost simple group. Combining the $d$ odd and $d$ even cases, it suffices to check that the theorem is true for $G_0 \cong PSU(4,2)$, which is easy to do in MAGMA.  

The cases when $x^3$ is a graph automorphism in an orthogonal group have already been considered, and the cases of graph automorphisms of exceptional groups are verified in MAGMA in the same way as with the other cases.
 \end{proof}
 This completes the proof of Theorem \ref{6}. We now prove Corollary \ref{cor:to9}.
 \begin{proof}[Proof of Corollary \ref{cor:to9}]
We may assume that $G$ has trivial solvable radical and that $x$ has order $9$. In particular, the Fitting subgroup $F(G)$ of $G$ is trivial and Lemma \ref{lem:1.1} applies. Let $F^{*}(G)$ be the generalized Fitting subgroup of $G$. Since $C_G(F^*(G))= \{1\}$ and $F(G)=\{1\}$, $x^3$ cannot centralize all of the components of $G$. So choose a component $L$ such that $x^3 \not\in C_G(L)$. If $x$ does not normalize $L$ then Lemma  \ref{lem:1.1} shows that there exists $g \in G$ such that $\langle x,x^g\rangle$  is not solvable. If $x$ does normalize $L$, then $x$ is an order $9$ element in the almost simple group $\langle x, L\rangle$ and Theorem \ref{6} implies the result.
\end{proof}

\begin{remark} \label{cor:to6}
In the introduction, we noted that the analogous result to Corollary \ref{cor:to9} for  order 6 elements is not true; for example if we take $G=S_5 \times PSL(3,3)$. However, in some sense, there are not many such examples. Let $G$ be such an example; we may assume that the solvable radical is trivial. Reasoning in the same way as in the proof of Corollary \ref{cor:to9}, if $x$ normalizes one of the components $L$, then $x$ must either act as an involution on $L$ or $L$ is one of the groups in Theorem \ref{thmA*} and $x$ acts as an element of order $3$ on $L$ (and as one of the exceptional elements in Theorem \ref{thmA*}). If $x$ does not normalize one of the components $L$ then $x^2$ must centralize $L$ by Lemma \ref{lem:1.1}. In particular, the orbits of the components under the action of $x$ must have length at most $2$.
\end{remark}

\end{document}